\documentclass[11pt,reqno]{amsart}
\usepackage{amssymb,amsmath}

\newcommand{\Bel}{\mathbf B}
\newcommand{\M}{\mathbf M}

\newcommand{\ci}[1]{_{ {}_{\scriptstyle #1}}}
\newcommand{\cii}[1]{_{ {}_{ #1}}}
\newcommand{\av}[2]{\langle #1\rangle\cii {#2}}

\newcommand{\const}{\operatorname{const}}
\newcommand{\sign}{\operatorname{sign}}
\newcommand{\vf}{\varphi}
\newcommand{\half}{\frac12}
\newcommand{\ve}{\varepsilon}
\newcommand{\R}{{\mathbb R}}
\newcommand{\D}{{\mathcal D}}
\newcommand{\J}{\mathcal J}

\newtheorem{theorem}{Theorem}
\newtheorem{lemma}[theorem]{Lemma}
\newtheorem{cor}[theorem]{Corollary}
\newtheorem*{prop}{Proposition}

\input epsf.sty

\begin{document}

\thispagestyle{empty}

\title[Burkholder's function]{{Burkholder's function\\ via Monge--Amp\`ere equation}}
\author{Vasily Vasyunin}\address{Vasily Vasyunin, V.A. Steklov. Inst., Fontanka 27, St. Petersburg, 191023, Russia,
{\tt vasyunin@pdmi.ras.ru}}
\author{Alexander Volberg}\address{Alexander Volberg, Dept. of  Math., MSU,
{\tt volberg@math.msu.edu}}

\thanks{Research of the first author was partially supported by RFBR grant 08-01-00723. The second author is supported by NSF grant DMS 0758552}

\maketitle

\begin{abstract}
We will show how to get Burkholder's  function from~\cite{Bu1} by using
Monge-Amp\`ere equation. This method is quite different from those in the
series of Burkholder's papers~\cite{Bu1}--\cite{Bu7}.
\end{abstract}

\section{Introduction}

Bellman function method in Harmonic Analysis was introduced by Burkholder for
finding the norm in $L^p$ of the Martingale transform. Later it became clear
that the scope of the method is quite wide.

The technique, originated in Burkholder's papers~\cite{Bu1}--\cite{Bu7}, can be
credited for helping to solve several old Harmonic Analysis problems and for
unifying approach to many others. In the first category one would name the
(sharp weighted) estimates of such classical operators as  the
Ahlfors--Beurling transform (Banuelos--Wang~\cite{BaWa1},
Banuelos--Janakiraman~\cite{BaJa1}, Banuelos--Mendez~\cite{BaMH},
Na\-za\-rov--Volberg~\cite{NV1}, Petermichl--Volberg~\cite{PV},
Dragicevic--Volberg~\cite{DV2}) and the Hilbert and Riesz transforms
(Petermichl~\cite{P1}, \cite{P2}). In the second category one can name all kind
of dimension free estimates of weighted and unweighted Riesz transforms (see a
vast literature in~\cite{DV1}--\cite{DV3}). Roughly, Bellman function method
makes apparent the hidden scaling properties of a given Harmonic Analysis
problem. Conversely, given a Harmonic Analysis problem with certain scaling
properties one can (formally) associate with is a non-linear PDE, the so-called
Bellman equation of the problem.

Let us recall to the reader that in the series of papers~\cite{Bu},
\cite{Bu1}--\cite{Bu7} Donald Burkholder investigated Martingale transform and
gave the sharp bounds on this operator in various settings--but by similar
methods. The methods were so novel and powerful that the influence of these
articles will be felt for many years to come. The novelty was a key. One of the
leading mathematician working in the domain of Harmonic Analysis  told the
second author that these papers of Burkholder ``spin his head".  In the book of
Daniel Strook~\cite{Str} many pages are devoted to the technique developed by
Burkholder in the abovementioned series of papers, and the reader can sense the
same feeling. It is explained in~\cite{Str} that the simplest way to understand
the sharp estimates of Martingale transform obtained by Burkholder is to
operate with one of the so-called Burkholder's function:
\begin{equation}
\label{up}
u_p(x,y) = p(1-\frac1{p^*})^{p-1}( |y|-(p^*-1)|x|)(|x|+|y|)^{p-1}\,,
\end{equation}
here $p^*:=\max (p, \frac{p}{p-1})$, $1<p<\infty\,.$

However, the main question is of course how to get this function? Where did it
come from? These questions are asked in~\cite{Str} as well. Of course,
Burkholder explains in many details the way this function (and several of its
relatives) are obtained. It is almost (but not quite) the least bi-concave
majorant of function
\begin{equation}
\label{op} |y|^{p-1}-(p^*-1)^p|x|^p\,.
\end{equation}
It is obtained by solving a certain PDE and performing certain manipulations
with the solution after that. The reader will find much more about $u_p$ after
reading this article, in particular in Section~\ref{shortcut}.

But it seems like the same questions persist even after this explanation. And a
new question can appear: how wide is the applicability of the technique that
Burkholder elaborated  in~\cite{Bu1}--\cite{Bu7}? There is a vague feeling that
the area of applicability is quite wide. To make this feeling more precise one
should look at the function above closer and see that it is a creature from
another universe, which, initially, does not have too much in common with
Harmonic Analysis. Burkholder function is a natural dweller of the area called
Stochastic Optimal Control. It is a solution of a corresponding Bellman
equation (or a dynamic programming equation) but in the setting, when the
differential equations subject to control are not the usual ones. They are
stochastic differential equations.
The reader can find some notes on this in~\cite{VoEcole}, \cite{NTV2}, \cite{VaVo},
\cite{VaVo2},~\cite{SlSt}. These notes explain why Stochastic Optimal Control
is the right tool to work with a certain class of Harmonic Analysis problems.
On the other hand, Stochastic Optimal Control problems generically can be
reduced to solving a so-called Bellman PDE (and proving the so-called
``verification theorems'', but this is a second task). Bellman PDEs belong to
the class of fully non-linear PDEs. Often they are PDEs of Monge--Amp\`ere
type. In the present article we would like to show the reader  how to obtain
Burkhloder functions (the one above and others from~\cite{Bu1}--\cite{Bu7}) by
reducing the search for them to solving certain Monge--Amp\`ere equations. The
scope of the application of the methods of Stochastic Optimal Control to
Harmonic Analysis proved to be quite large. After Burkholder the first
systematic application of this technique appeared in 1995 in the first preprint
version of~\cite{NTV1}. It was vastly developed in~\cite{NT} and in (now)
numerous  papers that followed. A small part of this literature can be found in
the bibliography below.


\section{Notations and definitions}

We shall say that an interval $I$ and a pair of positive numbers $\alpha^\pm_I$
such that $\alpha^++\alpha^-=1$ generate a pair of subintervals $I^+$ and $I^-$
if $|I^\pm|=\alpha^\pm_I|I|$ ($|I|$ means the length of $I$) and
$I=I^-\cup{I^+}$. For a given interval $J$ the symbol $\J=\J(\alpha)$ will
denote the families of subintervals of $J$ such that
\begin{itemize}
\item{$J\in\J$}\,;
\item{if $I\in\J$ then $I^\pm\in\J$\,.}
\end{itemize}
For a special choice if all $\alpha^\pm_I=\half$ we get the dyadic family $\J=\D$.
Every family $\J$ has its own set of Haar functions:
\begin{equation*}
\forall I\in\J\qquad h\cii I(t)=
\begin{cases}
\phantom{-}\sqrt{\frac{\alpha^+_I}{\alpha^-_I|I|}}\quad&\text{if}\ t\in I_-,\\
-\sqrt{\frac{\alpha^-_I}{\alpha^+_I|I|}}\quad&\text{if}\ t\in I_+.
\end{cases}
\end{equation*}
If the family $\J$ is such that that the maximal length of the interval of
$n$-th generation (i.e., after splitting the initial interval $J$ into $2^n$
parts) tends to $0$ as $n\to\infty$, the Haar family forms an orthonormal basis
in the space $L^2(J)\ominus\{\const\}$.
\medskip

For a function $f\in L^1(I)$ the symbol $\av fI$ means the average of $f$ over
the interval $I$:
$$
\av fI=\frac1{|I|}\int_I f(t)dt\,.
$$

\noindent{\bf Definition.} Fix a real $p$, $1<p<\infty$, and let $p'=\frac
p{p-1}$, $p^*=\max\{p,p'\}$. Introduce the following domain in $\R^3$:
\begin{equation*}
\Omega=\Omega(p)=\{x=(x_1,x_2,x_3):\ x_3\ge0,\ |x_1|^p\le x_3\}\,.
\end{equation*}
For a fixed partition $\J$ of an interval $J$ we define two function on this
domain
\begin{equation*}
\gathered
\Bel_{\max}(x)=\Bel_{\max}(x;p)=\sup_{f,\,g}\bigl\{\av{|g|^p}J\bigr\},
\\
\Bel_{\min}(x)=\Bel_{\min}(x;p)=\inf_{f,\,g}\bigl\{\av {|g|^p}J \bigr\},
\endgathered
\end{equation*}
where the supremum is taken over all functions $f$, $g$ from $L^p(J)$ such that
$\av fJ=x_1$, $\av gJ=x_2$, $\av{|f|^p}J=x_3$, and
$|(f,h\cii{I})|=|(g,h\cii{I})|$. We shall refer to any such pair of functions
$f$, $g$ as to an admissible pair. When $|(f,h\cii I)|=|(g,h\cii I)|$ happens
for all dyadic intervals inside $J$ {\bf  we call $g$ a Martingale transform of
$f$}. We shall call $\Bel_{\max}(x)$ (and $\Bel_{\min}(x)$) the Bellman
functions of the problem of finding the best constant for the Martingale
transform inequality:
\begin{equation}
\label{p1}
|\av gJ|\le |\av fJ|\Rightarrow \langle |g|^p\rangle_J\le C(p) \langle |f|^p\rangle_J\,.
\end{equation}
This best constant was found by Burkholder: $C(p)=(p^*-1)^p,\, p^*:=\max (p, \frac{p}{p-1})$.

\noindent{\bf Remark 1.} It is amazing that there is no proof that would find
this $C(p)$ without finding the function of $3$ variables $\Bel_{\max}(x)$ or
some of its relatives (like, for example, $u_p$ from~\eqref{up}).

\medskip

\noindent{\bf Remark 2.} Burkholder proved that the functions $\Bel$ do not
depend on the initial interval $J$ and on a specific choice of its partition.
Below we work only with dyadic partitions.

\medskip

\noindent{\bf Remark 3.} In the case $p=2$ the Bellman function are evident:
\begin{equation*}
\Bel_{\max}(x)=\Bel_{\min}(x)=x_2^2+x_3-x_1^2\,.
\end{equation*}
Indeed, since
\begin{equation*}
\|f\|_2^2=|J|x_3=|J|x_1^2+\sum_{I\in\J}|(f,h\cii I)|^2\,,
\end{equation*}
we have
\begin{align*}
\av{|g|^2}J&=\frac1{|J|}\|g\|_2^2=x_2^2+\frac1{|J|}\sum_{I\in\J}|(g,h\cii I)|^2
\\
&=x_2^2+\frac1{|J|}\sum_{I\in\J}|(f,h\cii I)|^2=x_2^2+x_3-x_1^2\,.
\end{align*}

\medskip

Define the following function on $\mathbb R^2_+=\{z=(z_1,z_2)\colon z_i>0\}$:
\begin{equation}
\label{Fp}
F_p(z_1,z_2)\!=\!
\begin{cases}
\big[z_1^p-(p^*-1)^p z_2^p\big],\,\, \text{if }
z_1\le(p^*\!\!-\!1)z_2\,,
\\
p(1-\frac{1}{p^*})^{p-1}(z_1\!+\!z_2)^{p-1}\big[z_1\!-\!(p^*\!\!-\!1)z_2\big],
\,\, \text{if } z_1\ge(p^*\!\!-\!1)z_2\,.
\end{cases}
\end{equation}
Note for for $p=2$ the expressions above are reduced to
$F_2(z_1,z_2)=z_1^2-z_2^2$.

\section{The main result}

Now we are ready to state the main result:

\begin{theorem}
\label{t1} The equation
$F_p(|x_1|,|x_2|)=F_p(x_3^{\frac1p},\Bel_{\phantom3}^{\frac1p})$ determines
implicitly the function $\Bel=\Bel_{\min}(x;p)$ and the equation
$F_p(|x_2|,|x_1|)=F_p(\Bel_{\phantom3}^{\frac1p},x_3^{\frac1p})$ determines
implicitly the function $\Bel=\Bel_{\max}(x;p)$.
\end{theorem}

\noindent{\bf Remark.} The reader can take a look at formulae (5.23)--(5.27) on
page 660 of~\cite{Bu1} and recognize that this is how Burkholder describes
$\Bel_{\max}$. The same is true for $\Bel_{\min}$.

\section{How to find Bellman functions}

We start from deducing the main inequality for Bellman functions. Introduce new
variables $y_1=\half(x_2+x_1)$, $y_2=\half(x_2-x_1)$, and $y_3=x_3$. In terms
of the new variables we define a function $\M$,
\begin{equation*}
\M(y_1,y_2,y_3)=\Bel(x_1,x_2,x_3)=\Bel(y_1-y_2,y_1+y_2,y_3)\,,
\end{equation*}
on the domain
\begin{equation*}
\Xi=\{y=(y_1,y_2,y_3)\colon\ y_3\ge0,\ |y_1-y_2|^p\le y_3\}\,.
\end{equation*}

Since the point of the boundary $x_3=|x_1|^p$ ($y_3=|y_1-y_2|^p$) occurs for
the only constant test function $f=x_1$ (and therefore then $g=x_2$ is a
constant function as well), we have
\begin{equation*}
\Bel(x_1,x_2,|x_1|^p)=|x_2|^p\,,
\end{equation*}
or
\begin{equation}
\M(y_1,y_2,|y_1-y_2|^p)=|y_1+y_2|^p\,.
\label{bc}
\end{equation}

Note that the function $\Bel$ is even with respect of $x_1$ and $x_2$, i.e.,
\begin{equation*}
\Bel(x_1,x_2,x_3)=\Bel(-x_1,x_2,x_3)=\Bel(x_1,-x_2,x_3)\,.
\end{equation*}
It follows from the definition of $\Bel$ if we consider the test functions
$\tilde f=-f$ for the first equality and $\tilde g=-g$ for the second one. For
the function $\M$ this means that we have the symmetry with respect to the
lines $y_1=\pm y_2$
\begin{equation}
\M(y_1,y_2,y_3)=\M(y_2,y_1,y_3)=\M(-y_1,-y_2,y_3)\,.
\label{symmetry}
\end{equation}
Therefore, it is sufficient to find the function $\Bel$ in the domain
\begin{equation}
\label{omegaplus}
\Omega_+=\Omega_+(p)=\{x=(x_1,x_2,x_3):\ x_i\ge0,\ |x_1|^p\le x_3\}\,,
\end{equation}
or the function $\M$ in the domain
\begin{equation}
\label{Xplus}
\Xi_+=\{y=(y_1,y_2,y_3)\colon\ y_1\ge0,\ -y_1\le y_2\le y_1,\ (y_1-y_2)^p\le y_3\}\,.
\end{equation}
Then we get the solution in the whole domain by putting
$$\Bel(x_1,x_2,x_3)=\Bel(|\,x_1|,|\,x_2|,x_3)\,.$$

Due to the symmetry~\eqref{symmetry} we have the following boundary conditions
on the ``new part'' of the boundary $\partial\Xi_+$:
\begin{equation}
\label{bc1}
\begin{aligned}
\frac{\partial\M}{\partial y_1}=\frac{\partial\M}{\partial y_2}\qquad&\text{on
the hyperplane }y_2=y_1\,,
\\
\frac{\partial\M}{\partial y_1}=-\frac{\partial\M}{\partial y_2}\qquad&\text{on
the hyperplane }y_2=-y_1\,.
\end{aligned}
\end{equation}

If we consider the family of test functions $\tilde f=\tau f$, $\tilde g=\tau
g$ together with $f$ and $g$ we come to the following homogeneity condition
\begin{equation*}
\Bel(\tau x_1,\tau x_2,\tau^p x_3)=\tau^p\Bel(x_1,x_2,x_3)\,,
\end{equation*}
or
\begin{equation*}
\M(\tau y_1,\tau y_2,\tau^p y_3)=\tau^p\M(y_1,y_2,y_3)\,.
\end{equation*}
We shall use this property in the following form: take derivative with respect
to $\tau$ and put $\tau=1$
\begin{equation}
y_1\frac{\partial\M}{\partial y_1}+y_2\frac{\partial\M}{\partial
y_2}+py_3\frac{\partial\M}{\partial y_3}=p\M(y_1,y_2,y_3)\,.
\label{homogen}
\end{equation}

Let us fix two points $x^\pm\in\Omega$ such that $|x^+_1-x^-_1|=|x^+_2-x^-_2|$,
for the corresponding points $y^\pm\in\Xi$ this means that either
$y^+_1=y^-_1$, or $y^+_2=y^-_2$. Then for an arbitrarily small number $\ve>0$
by the definition of the Bellman function $\Bel=\Bel_{\max}$ there exist two
couples of test functions $f^\pm$ and $g^\pm$ on the intervals $I^\pm$ such
that $\av{f^\pm}{I^\pm}=x^\pm_1$, $\av{g^\pm}{I^\pm}=x^\pm_2$,
$\av{|f^\pm|^p}{I^\pm}=x^\pm_3$, and $\av{|g^\pm|^p}{I^\pm}\ge\Bel(x^\pm)-\ve$.
On the interval $I=I^+\cup I^-$ we define a pair of test functions $f$ and $g$
as follows $f|I^\pm=f^\pm$, $g|I^\pm=g^\pm$. This is a pair of test functions
that corresponds to the point $x=\alpha^+x^++\alpha^-x^-$, where
$\alpha^\pm=|I^\pm|/|I|$, because the property $|x^+_1-x^-_1|=|x^+_2-x^-_2|$
means $|(f,h\ci{I})|=|(g,h\ci{I})|$. This yields
\begin{equation*}
\Bel(x)\ge\av{|g|^p}I=\alpha^+\av{|g^+|^p}I^++\alpha^-\av{|g^-|^p}I^-\ge
\alpha^+\Bel(x^+)+\alpha^-\Bel(x^-)-\ve\,.
\end{equation*}
Since $\ve$ is arbitrary we conclude
\begin{equation}
\label{main-max}
\Bel(x)\ge\alpha^+\Bel(x^+)+\alpha^-\Bel(x^-)\,.
\end{equation}
For the function $\Bel=\Bel_{\min}$ we can get in a similar way
\begin{equation}
\label{main-min}
\Bel(x)\le\alpha^+\Bel(x^+)+\alpha^-\Bel(x^-)\,.
\end{equation}

Recall that this is not quite concavity (convexity) condition, because we have
the restriction $|x^+_1-x^-_1|=|x^+_2-x^-_2|$. But in terms of the function
$\M$
\begin{equation*}
\begin{gathered}
\M_{\max}(y)\ge\alpha^+\M_{\max}(y^+)+\alpha^-\M_{\max}(y^-)\,,
\\
\M_{\min}(y)\le\alpha^+\M_{\min}(y^+)+\alpha^-\M_{\min}(y^-)\,,
\end{gathered}
\end{equation*}
when either $y_1=y^+_1=y^-_1$, or $y_2=y^+_2=y^-_2$, we indeed have the
concavity (convexity) of the function $\M$ with respect to $y_2$, $y_3$ under a
fixed $y_1$, and with respect to $y_1$, $y_3$ under a fixed $y_2$.

Since the domain is convex, under the assumption that the function $\Bel$ are
sufficiently smooth these conditions of concavity (convexity) are equivalent to
the differential inequalities
\begin{equation}
\begin{pmatrix}
\M_{y_1y_1} & \M_{y_1y_3}\\
\M_{y_3y_1} & \M_{y_3y_3}
\end{pmatrix} \le 0\,,\quad
\begin{pmatrix}
\M_{y_2y_2} & \M_{y_2y_3}\\
\M_{y_3y_2} & \M_{y_3y_3}
\end{pmatrix} \le 0\,,\quad
\forall y\in \Xi\,,
\label{main_max}
\end{equation}
for $\M=\M_{\max}$ (here
$\M_{y_iy_j}$ stand for the partial derivatives
$\displaystyle\frac{\partial^2\M}{\,\partial y_i\partial y_j}$) and
\begin{equation}
\begin{pmatrix}
\M_{y_1y_1} & \M_{y_1y_3}\\
\M_{y_3y_1} & \M_{y_3y_3}
\end{pmatrix} \ge 0\,,\quad
\begin{pmatrix}
\M_{y_2y_2} & \M_{y_2y_3}\\
\M_{y_3y_2} & \M_{y_3y_3}
\end{pmatrix} \ge 0\,,\quad
\forall y\in \Xi\,,
\label{main_min}
\end{equation}
for $\M=\M_{\min}$.

Extremal properties of the Bellman function requires for one of matrices
in~\eqref{main_max} and~\eqref{main_min} to be degenerated. So we arrive at the
Monge--Amp\`ere equation:
\begin{equation}
\label{MA}
\M_{y_iy_i}\M_{y_3y_3}=(\M_{y_iy_3})^2
\end{equation}
either for $i=1$ or for $i=2$. To find a candidate $M$ for the role of the true
Bellman function $\M$ we shall solve this equation. After finding this solution
we shall prove that $M=\M$.

The method of solving homogeneous Monge--Amp\`ere equation is described, for
example, in~\cite{VaVo}, \cite{VaVo2},~\cite{SlSt}. In particular we know that the solution of the
Monge--Amp\`ere equation has to be of the form
\begin{equation}
M=t_iy_i+t_3y_3+t_0\,,
\label{linear_form}
\end{equation}
where $t_k=M_{y_k}$, $k=1,2,3$. The solution $M$ is linear along the lines (let
us call them {\it extremal trajectories})
\begin{equation}
y_idt_i+y_3dt_3+dt_0=0\,.
\label{traject}
\end{equation}
One of the ends of the extremal trajectory has to be a point on the boundary
$y_3=|y_1-y_2|^p$, where constant functions are the only test functions
corresponding to these points. Denote this point by $U=(y_1,u,(y_1-u)^p)$. Note
that we write $(y_1-u)^p$ instead of $|y_1-u|^p$ because the domain $\Xi_+$ is
under consideration. For the second end of the extremal trajectory we have four
possibilities
\begin{itemize}
\item[1)] it belongs to the same boundary $y_3=(y_1-y_2)^p$;
\item[2)] it is at infinity $(y_1,y_2,+\infty)$, i.e., the extremal lines goes
parallel to the $y_3$-axis;
\item[3)] it belongs to the boundary $y_2=y_1$;
\item[4)] it belongs to the boundary $y_2=-y_1$.
\end{itemize}

The first possibility gives us no solution.  Namely, we have the following
\begin{prop}
If $p\ne2$, then the function $\Bel_{\max}$ cannot be equal to $B(x)=M(y)$,
where $M$ is the solution of the Monge--Amp\`ere equation~\eqref{MA} such that
one of its extremal trajectory is of type $1)$ above. The same claim holds for
$\Bel_{\min}$.
\end{prop}

\begin{proof}
To check this it is sufficient to verify that the test functions of the type
$\alpha+\beta h\ci{I}(t)$ cannot be an extremal function of our problem with
the only exception of $p=2$, when the situation is trivial:
$\Bel_{\max}(x)=\Bel_{\min}(x)=x_3+x_2^2-x_1^2$, and any pair of test function
is extremal. We will show that the Bellman functions being solution of the
homogeneous Monge-Amp\`ere equation cannot be linear on a chord $[x^-, x^+]$
connecting two points $x^\pm$ on the boundary $\partial\Omega$, i.e. such a
chord cannot be an extremal trajectory of our Monge-Amp\`ere equation.

We assume now that two points $x^\pm\in\Omega_+$ such that
$|x_1^+-x_1^-|=|x_2^+-x_2^-|$ are on the boundary $x_3^\pm=(x_1^\pm)^p$ and
$x=\half(x^++x^-)$. We need to show that $\half((x_2^+)^p+(x_2^-)^p)$ can be
the value neither of $\Bel_{\max}(x)$ nor of $\Bel_{\min}(x)$.

Without lost of generality we may assume that $x_1^+>x_1^-$. Let us denote
$a:=\half(x_1^+-x_1^-)$ then
$$
x_1^\pm=x_1\pm a\qquad x_2^\pm=x_2\pm\sigma a\,,
$$
where $\sigma=\pm1$ depending on the direction of our chord: it can be either
in the plane $x_1-x_2=\const$ (and then $\sigma=1$) or in the plane
$x_1+x_2=\const$ (and then $\sigma=-1$). The pair of the test functions $f,g$
on $I=[0,1]$ that gives the value
$$
A(a)=:\av{|g|^p}I=\half\big(|\,x_2^+|^p+|\,x_2^-|^p\big)=
\half\big(|x_2+a|^p+|x_2-a|^p\big)
$$
is
$$
f=x_1+ah\ci{I}\qquad g=x_2+\sigma ah\ci{I}\,.
$$
First of all we assume that $x^\pm\in\Omega_+$, but $x^\pm_i\ne0$ i.e. $x_1\pm
a>0$, $x_2\pm a>0$, because if one of $x_i^\pm$ is zero we are in the cases
either $3)$ or $4)$ listed before Proposition. Our aim will be to find another
pair of test functions $\tilde f,\tilde g$ corresponding to the same point $x$,
but with $\av{|\tilde g|^p}I$ either bigger than $A(a)$ (and then $A(a)$ cannot
be the value of $\Bel_{\max}(x)$) or less than $A(a)$ (and then $A(a)$ cannot
be the value of $\Bel_{\min}(x)$).

Let as make here two remarks. First, we see that the expression $A(a)$ does not
depend on the direction $\sigma$. Therefore, the construction of the desired
$\tilde f,\tilde g$ ensures us that the point $x$ cannot be the center of an
extremal trajectory with two ends on $\partial\Omega$ in any direction
$\sigma=\pm1$. Secondly, we note that it is not obligatory to look for $\tilde
f,\tilde g$ for all $a\in(0,\max\{x_1,x_2\})$, it sufficient to do this for
small values of $a/x_2$.

Indeed, suppose that the chord $L=[x^-,x^+]$ represents an extremal trajectory
of the corresponding Monge--Amp\`ere equation. Let us consider the ``crescent''
between the chord $L$ and the boundary $\partial\Omega$. It should be filled in
by chords on which $\Bel(x)$ are linear (this is the property of the solutions
of the homogeneous Monge--Amp\`ere equation expressed in Pogorelov's theorem,
see~\cite{Pog}). Among these chords we can take one (say $\tilde L=[\tilde
x^-,\tilde x^+]$) of arbitrarily small length (arbitrarily small value of
$\tilde a/\tilde x_2$). Therefore, we can work with a new point $\tilde x$ and
new chord $\tilde L$: would we show that $\tilde L$ cannot be an extremal
trajectory, the chord $L$ could not be one either.

To construct the desired pair $\tilde f,\tilde g$ we use the following family
of functions $\phi_s$ equal to $1$ on $[0, 1/2-s]\cup [1-s,1]$ and to $-1$ on
$(1/2-s, 1-s)$. Let us note that all $\phi_s$ have the same distribution function as  $h\ci{I}$,
and
$$
(\phi_s, h\ci{I}) = 1-4s\,.
$$
We are interested in $\phi:=\phi_{1/8}$. Then
$$
(\phi, h\ci{I}) = \half\,,
$$
and $\psi = \phi -h\ci{I}$ has
$$
(\psi, h\ci{I}) = -\half\,,\qquad(\psi, h\ci{\!J})=(\phi, h\ci{\!J})
$$
for all other dyadic $J$. So $\psi$ is a martingale transform of $\phi$ (it is
equal to $0$ on $[0, 3/8]\cup[1/2, 7/8]$ and to $\pm2$ on two intervals
$(7/8,1)$, $(3/8, 1/2)$) and we can examine the pair
$$
\tilde f=x_1+a\phi,\qquad\tilde g=x_2+a\psi\,.
$$
Since $\tilde f$ and $f$ have the same distribution function, we have $\av{\,|\tilde f|^p}I =
\av{\,|f|^p}I=x_3$, i.e., $\tilde f,\tilde g$ is a pair of test functions
corresponding to the same point $x$.

To investigate the difference $\av{\,|\tilde g|^p}I-\av{\,|g|^p}I$ we use the
function
$$
\lambda_p(\alpha):=\frac18\big((1+2\alpha)^p+(1-2\alpha)^p\big)+\frac34-
\big((1+\alpha)^p + (1-\alpha)^p\big)\,.
$$
Since
$$
\lambda_p(\alpha)=\frac18p(p-1)(p-2)(p-3)\alpha^4+O(\alpha^6)\,,
$$
we have $\lambda_p(\alpha)>0$ for small $\alpha$ if $1<p<2$ or $p>3$ and
$\lambda_p(\alpha)<0$ for small $\alpha$ if $2<p<3$. Recall that
$$
\tilde g(t)=
\begin{cases}
\quad x_2\quad&0<t<\frac38
\\
x_2-2a&\frac38<t<\half
\\
\quad x_2&\half<t<\frac78
\\
x_2+2a&\frac78<t<1\,,
\end{cases}
$$
therefore, $\av{\,|\tilde g|^p}I=
\frac18\big((x_2+2a)^p+(x_2-2a)^p\big)+\frac34x_2^p$ and
$$
\av{\,|\tilde g|^p}I-\av{\,|g|^p}I=x_2^p\lambda\big(\frac a{\,x_2}\big).
$$
For small $\alpha$ we have a desired example for $\Bel_{\max}$ if
$p\in(1,2)\cup(3,\infty)$ (because $\lambda_p>0$ and for $\Bel_{\min}$ if
$p\in(2,3)$ (because $\lambda_p<0$).

Now we interchange in a sense the roles of $\tilde f$ and $\tilde g$: instead of
$a$ we take a new parameter, say $\tilde a$, and put
$$
\tilde f=x_1+\tilde a\psi,\qquad\tilde g=x_2+\tilde a\phi\,.
$$
Since we have now
$$
\av{\,|\tilde f|^p}I-\av{\,|f|^p}I=x_2^p\lambda\big(\frac {\tilde a}{\,x_1}\big)
$$
and the function
$$
t\mapsto \av{\,|x +t\psi|^p}I=\frac18(|x-2t|^p+|x+2t|^p)+\frac34|x|^p
$$
is increasing in $t>0$ from $|x|^p$ till infinity, we can find $\tilde a$,
$\tilde a>a$ for $p\in(2,3)$ and $\tilde a<a$ for $p\in(1,2)\cup(3,\infty)$,
such that
$$
\av{\,|\tilde f|^p}I=\av{\,|f|^p}I=x_3.
$$
For this $\tilde a$ we get the desired pair of test function, because the
function
$$
t\mapsto \av{|x+t\phi|^p}I= \frac12 (|x-t|^p +|x+t|^p)
$$
is also increasing in $t>0$, and therefore, we have
$$
\av{\,|\tilde g|^p}I=\av{\,|x_2+\tilde a\phi|^p}I>\av{\,|x_2+\tilde a\phi|^p}I=
\av{\,|x_2+\tilde ah\ci{I}|^p}I=\av{\,|g|^p}I
$$
if $p\in(2,3)$ and the opposite inequality if $p\in(1,2)\cup(3,\infty)$ (and
$\tilde a<a$).

This construction failed for $p=3$ because $\lambda_3(\alpha)=0$ for all
$\alpha\in(0,\half)$. To avoid this difficulty we modify the function $\psi$,
namely, we take $\psi=\phi+h\ci{I^+}/\sqrt2$, i.e.,
$$
\psi(t)=
\begin{cases}
\phantom{-}1,\quad&0<t<\frac38\,,
\\
-1,&\frac38<t<\half\,,
\\
\phantom{-}0,&\half<t<\frac34\,,
\\
-2,&\frac34<t<\frac78\,,
\\
\phantom{-}0,&\frac78<t<1\,.
\end{cases}
$$
The function $\psi$ is a martingale transform of $\phi$, since
$$
(\,\phi,h\ci{I^+}\!)=-\frac1{2\sqrt2}\,,\qquad(\,\psi,h\ci{I^+}\!)=\frac1{2\sqrt2}\,.
$$
Now we put
$$
\tilde f=x_1+a\phi,\qquad\tilde g=x_2\pm a\psi\,.
$$
As before we have $\av{\,|\tilde f|^3}I=\half\big((x-a)^3+(x+a)^3\big)
=\av{\,|f|^3}I$, but
$$
\av{\,|\tilde g|^3}I=x_2^3+6x_2a^2\mp\frac34a^3=\av{\,|g|^3}I\mp\frac34a^3\,,
$$
therefore, by choosing the sign in the definition of $\tilde g$ we are able to
increase as well to decrease the value $\av{\,|g|^3}I$, hence it is neither the
value of $\Bel_{\max}(x)$ nor $\Bel_{\min}(x)$. Proposition is completely
proved.
\end{proof}

Now we check the second possibility among the possibilities 1)--4) listed right
before the Proposition. Since the extremal line is parallel to the $y_3$-axis,
the Bellman function has to be of the form
\begin{equation*}
M(y)=A(y_1,y_2)+C(y_1,y_2)y_3\,.
\end{equation*}
Any pair of inequalities both~\eqref{main_max} and~\eqref{main_min} implies
$M_{y_iy_i}M_{y_3y_3}-(M_{y_iy_3})^2\ge0$. Since $M_{y_3y_3}=0$, this yields
$M_{y_iy_3}=\frac{\partial C}{\partial y_i}=0$, i.e., $C$ is a constant. From
the boundary condition~\eqref{bc} we get
\begin{equation*}
A(y_1,y_2)+C(y_1-y_2)^p=(y_1+y_2)^p\,,
\end{equation*}
whence
\begin{equation*}
A(y_1,y_2)=(y_1+y_2)^p-C(y_1-y_2)^p\,,
\end{equation*}
and
\begin{equation}
M(y)=(y_1+y_2)^p+C(y_3-(y_1-y_2)^p)\,, \label{solut2M}
\end{equation}
or
\begin{equation}
B(x)=x_2^p+C(x_3-x_1^p)\,. \label{solut2B}
\end{equation}

Let us note that this solution cannot satisfy necessary conditions in the whole
domain $\Xi_+$ except the case $p=2$. The constant $C$ must be positive
(otherwise the extremal lines cannot tend to infinity along $y_3$-axes, because
$M$ must be a nonnegative function). Therefore, the straight line
\begin{equation*}
y_1+y_2=C^{\frac1{p-2}}(y_1-y_2),\,\,\text{or}\,\, x_2=C^{\frac1{p-2}}y_1
\end{equation*}
splits $\Xi_+$ in two subdomains, in one of which the derivatives
\begin{equation*}
\frac{\partial^2 M}{\partial y_1^2}=\frac{\partial^2 M}{\partial y_2^2}
=p(p-1)\Big((y_1+y_2)^{p-2}-C(y_1-y_2)^{p-2})\Big)
\end{equation*}
is positive (i.e., it could be a candidate for $\Bel_{\min}$), and in another
one is negative (i.e., it could be a candidate for $\Bel_{\max}$).

\medskip

Thus, this simple solution cannot give us the whole Bellman function and we
need to continue the consideration of the possibilities~3) and~4) (listed right
before the Proposition. Till now we have not fixed which of two matrices
in~\eqref{main_max} or in~\eqref{main_min} is degenerated, i.e., what is $i$ in
the Monge--Amp\`ere equation~\eqref{MA}, because for the vertical extremal
lines both these equations are fulfilled. Now, when considering possibility~3)
or~4), we need to investigate separately both Monge--Amp\`ere
equations~\eqref{MA}. We shall refer to these cases as~$3_i\!)$ and $4_i\!)$.

Let us start with simultaneous consideration of the cases~$3_1\!)$ and~$4_1\!)$
(we recall that this means that $y_2$ is fixed). We look for a function
\begin{equation*}
M=t_1y_1+t_3y_3+t_0
\end{equation*}
on the domain $\Xi_+$, which is linear along the extremal lines
\begin{equation*}
y_1dt_1+y_3dt_3+dt_0=0\,.
\end{equation*}

Now one end point of our extremal line $V=(v,y_2,(v-y_2)^p)$ belongs to the
boundary $y_3=|y_1-y_2|^p$ and the second end point $W=(|y_2|,y_2,w)$ is on the
boundary $y_1=|y_2|$, where we have boundary condition~\eqref{bc1}. Due to the
symmetry~\eqref{symmetry}, on the boundary $y_1=y_2$ (this means that our fixed
$y_2\ge 0$) we have
\begin{equation*}
\frac{\partial M}{\partial y_2}=\frac{\partial M}{\partial y_1}=t_1\,,
\end{equation*}
and
\begin{equation*}
\frac{\partial M}{\partial y_2}=-\frac{\partial M}{\partial y_1}=-t_1\,,
\end{equation*}
on the boundary $y_1=-y_2$ (this means that our fixed $y_2<0$). In both cases
\begin{equation*}
y_2\frac{\partial M}{\partial y_2}=y_1\frac{\partial M}{\partial y_1}=|y_2|t_1\,,
\end{equation*}
and therefore~\eqref{homogen}  and~\eqref{linear_form} imply
\begin{equation*}
2t_1|y_2|+pwt_3=pM(W)=pt_1|y_2|+pwt_3+pt_0\,,
\end{equation*}
whence
\begin{equation*}
t_0=\big(\frac2p-1\big)t_1|y_2|\,.
\end{equation*}
This gives the formula for $t_0(t_1)$ (remember that $y_2$ is fixed as we
consider the cases $3_1)$, $4_1)$ now). Thus, we get
\begin{equation}
M(y)=\Big[\,y_1+\big(\frac2p-1\big)|y_2|\Big]t_1+y_3t_3\,. \label{M1}
\end{equation}

Since $dt_0=\big(\frac2p-1\big)|y_2|\,dt_1$, the equation of the extremal
trajectories~\eqref{traject} takes the form
\begin{equation}
\Big[\,y_1+\big(\frac2p-1\big)|y_2|\Big]dt_1+y_3\,dt_3=0\,, \label{traject1}
\end{equation}
and we can rewrite~\eqref{M1} as follows
\begin{equation*}
M(y)=\Big(t_3-t_1\frac{dt_3}{dt_1}\Big)y_3\,.
\end{equation*}
We see that the expression $M(y)/y_3$ is constant along the trajectory and we
can find it evaluating at the point $V$, where the boundary
condition~\eqref{bc} is known:
\begin{equation}
M(y)=\left(\frac{v+y_2}{v-y_2}\right)^{\!p}\!\!y_3\,, \label{solut1}
\end{equation}
where $v=v(y_1,y_2,y_3)$ satisfies the following equation:
\begin{equation}
\frac{y_1+\big(\frac2p-1\big)|y_2|}{y_3}=\frac{v+\big(\frac2p-1\big)|y_2|}{(v-y_2)^p}\,,
\label{extr1}
\end{equation}
because the point $V=(v,y_2,(v-y_2)^p)$ is on the extremal
line~\eqref{traject1}. We even shall not check under what conditions
equation~\eqref{extr1} has a solution and when it is unique. Later we show that
in any case the function $M$ we have found cannot be the Bellman function we
are interested in, because neither condition~\eqref{main_max}
nor~\eqref{main_min} can be fulfilled: the matrix $\{M_{y_iy_j}\}_{i,j=2,3}$ is
neither negative definite nor positive definite. We postpone this verification,
because the calculation of the sign of the Hessian matrices is the same for
this solution and another solution of the Monge--Amp\`ere equation that
supplies us with the true Bellman function. And these calculations will be made
simultaneously a bit later. And now we only rewrite our solution in an implicit
form more convenient for calculation.

We introduce
\begin{equation}
\label{omega}
\omega:=\left(\!\frac{M(y)}{y_3}\!\right)^{\frac1p},
\end{equation}
then~\eqref{solut1} yields
\begin{equation}
\label{vomega}
v=\frac{\omega+1}{\omega-1}\,y_2\,.
\end{equation}
Since $v\ge0$ (in fact, recall that we consider now only $y$: $y_1 \ge |y_2|$ domain now, and that $v$
is just the first coordinate of the point $V=(v, y_2, (v-y_2)^p$ in this domain), we have
\begin{equation}
\label{signy2}
\sign y_2=\sign(\omega-1)\,.
\end{equation}
After substitution
of~\eqref{vomega} in~\eqref{extr1} we get
\begin{equation*}
\Big(\frac{2y_2}{\omega-1}\Big)^{\!p}\Big[y_1+\big(\frac2p-1\big)|y_2|\Big]=
y_3\Big[\frac{\omega+1}{\omega-1}y_2+\big(\frac2p-1\big)|y_2|\Big]
\end{equation*}
or
\begin{equation*}
2^p|y_2|^{p-1}\big[py_1+(2-p)|y_2|\big]=
y_3|\omega-1|^{p-1}\big[(\omega+1)p+(2-p)|\omega-1|\big]
\end{equation*}
For the case~$3_1\!)$ we have $y_2>0$ (i.e., $x_2>x_1$, we look for $\omega>1$
or $B>y_3$) and the latter equation can be rewritten in the initial coordinates
as follows
\begin{equation*}
(x_2-x_1)^{p-1}\big[x_2+(p-1)x_1]=
(B_{\phantom3}^{\frac1p}-x_3^{\frac1p})^{p-1}\big[B_{\phantom3}^{\frac1p}+(p-1)x_3^{\frac1p}\big]\,.
\end{equation*}
For the case $4_1\!)$ we have $y_2<0$ (i.e., $x_2<x_1$, we look for $\omega<1$
or $B<y_3$) and the equation takes the form
\begin{equation*}
(x_1-x_2)^{p-1}\big[x_1+(p-1)x_2]=
(x_3^{\frac1p}-B_{\phantom3}^{\frac1p})^{p-1}\big[x_3^{\frac1p}+(p-1)B_{\phantom3}^{\frac1p}\big]
\end{equation*}

Introduce the following function
\begin{equation*}
G(z_1,z_2)=(z_1+z_2)^{p-1}\big[z_1-(p-1)z_2]
\end{equation*}
defined on the half-plane $z_1+z_2\ge0$. Then in the case~$3_1\!)$ we have the
relation
\begin{equation}
\label{G31}
G(x_2,-x_1)=G(B_{\phantom3}^{\frac1p},-x_3^{\frac1p})\,,
\end{equation}
or
\begin{equation*}
G(y_2+y_1,y_2-y_1)=y_3G(\omega,-1)\,.
\end{equation*}
In the case~$4_1\!)$ we have
\begin{equation}
\label{G41}
G(x_1,-x_2)=G(x_3^{\frac1p},-B_{\phantom3}^{\frac1p})\,,
\end{equation}
or
\begin{equation*}
G(y_1-y_2,-y_1-y_2)=y_3G(1,-\omega)\,.
\end{equation*}

\medskip

Now we have to consider the Monge-Amp\`ere equation~\eqref{MA} in the
cases~$3_2\!)$ and~$4_2)$. This means that we fix $y_1$ now. Let us begin with
the cases~$3_2\!)$, when an extremal line starts at a point
$U=(y_1,u,(y_1-u)^p)$ on our parabola and ends at a point $W=(y_1,y_1,w)$.
Again, the symmetry condition at the point $W$ is
\begin{equation*}
\frac{\partial M}{\partial y_1}=\frac{\partial M}{\partial y_2}=t_2\,,
\end{equation*}
and the homogeneity condition~\eqref{homogen}  plus
condition~\eqref{linear_form} at $W$ yield
\begin{equation*}
2y_1t_2+pwt_3=pM(W)=py_1t_2+pwt_3+pt_0\,,
\end{equation*}
whence
\begin{equation*}
t_0=\big(\frac2p-1\big)y_1t_2\,,
\end{equation*}
and therefore
\begin{equation}
M(y)=\Big[\,y_2+\big(\frac2p-1\big)y_1\Big]t_2+y_3t_3\,. \label{M32}
\end{equation}

Since $dt_0=\big(\frac2p-1\big)y_1\,dt_2$, the equation of the extremal
trajectories takes the form
\begin{equation}
\Big[\,y_2+\big(\frac2p-1\big)y_1\Big]dt_2+y_3\,dt_3=0\,, \label{traject32}
\end{equation}
and we can rewrite~\eqref{M32} as follows
\begin{equation*}
M(y)=\Big(t_3-t_2\frac{dt_3}{dt_2}\Big)y_3\,.
\end{equation*}
Again, from here the expression $M(y)/y_3$ is constant along the trajectory and we can
find it evaluating at the point $U$, where we the boundary condition~\eqref{bc}
is known:
\begin{equation}
M(y)=\left(\frac{y_1+u}{y_1-u}\right)^{\!p}\!\!y_3\,, \label{solut32}
\end{equation}
where $u=u(y_1,y_2,y_3)$ can be found from~\eqref{traject32}:
\begin{equation}
\frac{y_2+\big(\frac2p-1\big)y_1}{y_3}=\frac{u+\big(\frac2p-1\big)y_1}{(y_1-u)^p}\,.
\label{extr32}
\end{equation}

We see that if our extremal line starts at point $U=(y_1,u,(y_1-u)^p$ on our
parabola $u=-\big(\frac2p-1\big)y_1$, then
$y_2=-\big(\frac2p-1\big)y_1=u=\const$, i.e., it is a line parallel to the
$x_3$-axes. This means that no extremal line that ends at the points of the
boundary $y_1=y_2$ can intersect the plane $y_2=-\big(\frac2p-1\big)y_1$. This
follows from the property that extremal trajectories do not intersect.
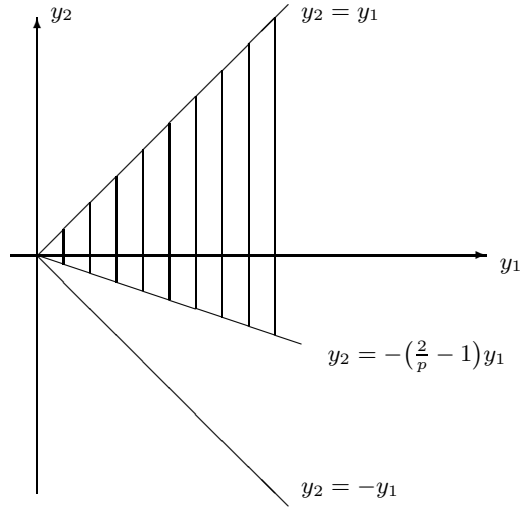
\begin{figure}[ht]
\begin{center}
\begin{picture}(200,200)
\thinlines
\put(20,10){\vector(0,1){180}}
\put(10,100){\vector(1,0){180}}
\put(20,100){\line(1,1){95}}
\put(20,100){\line(1,-1){95}}
\put(20,100){\line(3,-1){100}}
\put(30,110){\line(0,-1){13.3}}
\put(40,120){\line(0,-1){26.6}}
\put(50,130){\line(0,-1){40}}
\put(60,140){\line(0,-1){53.3}}
\put(70,150){\line(0,-1){66.6}}
\put(80,160){\line(0,-1){80}}
\put(90,170){\line(0,-1){93.3}}
\put(100,180){\line(0,-1){106.6}}
\put(110,190){\line(0,-1){120}}
\put(120,190){\footnotesize $y_2=y_1$}
\put(120,10){\footnotesize $y_2=-y_1$}
\put(130,60){\footnotesize $y_2=-\big(\frac2p-1\big)y_1$}
\put(150,80){\footnotesize $$}
\put(25,190){\footnotesize $y_2$}
\put(195,95){\footnotesize $y_1$}
\end{picture}
\caption{Acceptable sector for the case $3_2\!).$}
\label{3_2}
\end{center}
\end{figure}
Therefore, the starting points $U$ with $u\le-\big(\frac2p-1\big)y_1$ cannot be
acceptable for the case under consideration (since these trajectories do not
intersect the plane $y_2=-\big(\frac2p-1\big)y_1$, they cannot have the second
end point on $y_2=y_1$, see~Fig.~\ref{3_2}).

Let us check that equation~\eqref{extr32} has exactly
one solution $u=u(y_1,y_2,y_3)$ in the sector
$-\big(\frac2p-1\big)y_1<y_2<y_1$. Indeed, the function
\begin{equation*}
u\mapsto
y_3\Big[u+\big(\frac2p-1\big)y_1\Big]-(y_1-u)^p\Big[y_2+\big(\frac2p-1\big)y_1\Big]
\end{equation*}
is monotonously increasing for $u<y_1$ and it has the negative value
$-\big(\frac2p y_1)^p\Big[y_2+\big(\frac2p-1\big)y_1\Big]$ at the point
$u=-\big(\frac2p-1\big)y_1$ and the positive value $\frac2p y_1y_3$ at the
point $u=y_1$.

Now we rewrite the  solution~\eqref{solut32} in an implicit form using
notations~\eqref{omega}: $\omega:=\left(\!\frac{M(y)}{y_3}\!\right)^{\frac1p}$.
From~\eqref{solut32} we have
\begin{equation}
u=\frac{\omega-1}{\omega+1}\,y_1\,,
\label{uomega}
\end{equation}
therefore, from~\eqref{extr32} we obtain
\begin{equation*}
2^{-p}y_3(\omega +1)^{p-1}[p(\omega -1) +(2-p)(\omega +1)]=y_1^{p-1}\big[py_2+(2-p)y_1\big]
\end{equation*}
or
$$
2^{-p+1}y_3(\omega+1)^{p-1}
(\omega-p+1)=y_1^{p-1}\big[py_2+(2-p)y_1\big]\,,
$$
which is (using again  notations~\eqref{omega}:
$\omega:=\left(\!\frac{M(y)}{y_3}\!\right)^{\frac1p}$)
\begin{equation*}
(B_{\phantom3}^{\frac1p}+x_3^{\frac1p})^{p-1}\big[B_{\phantom3}^{\frac1p}-(p-1)x_3^{\frac1p}\big]=
(x_1+x_2)^{p-1}\big[x_2-(p-1)x_1\big]\,.
\end{equation*}
In terms of function $G$ this can be rewritten as follows
\begin{equation*}
G(x_2,x_1)=G(B_{\phantom3}^{\frac1p},x_3^{\frac1p})\,,
\end{equation*}
or
\begin{equation*}
G(y_1+y_2,y_1-y_2)=y_3G(\omega,1)\,.
\end{equation*}

\medskip

It remains to examine the possibility $4_2\!)$. Assume that an extremal line
starts at a point $U=(y_1,u,(y_1-u)^p)$ and ends at a point $W=(y_1,-y_1,w)$.
Again, the homogeneity property~\eqref{homogen} at the point $W$ and the
symmetry $\frac{\partial M}{\partial y_1}=-\frac{\partial M}{\partial
y_2}=-t_2$ yield
\begin{equation*}
-2y_1t_2+pwt_3=pM(W)=-py_1t_2+pwt_3+pt_0\,,
\end{equation*}
whence
\begin{equation*}
t_0=\big(1-\frac2p\,\big)y_1t_2\,,
\end{equation*}
and therefore
\begin{equation}
M(y)=\Big[y_2+(1-\frac2p\,)\,y_1\Big]t_2+y_3t_3\,. \label{M42}
\end{equation}

Since $dt_0=(1-\frac2p\,)y_1\,dt_2$, the equation of the extremal trajectories
takes the form
\begin{equation}
\Big[y_2+(1-\frac2p\,)\,y_1\Big]dt_2+y_3\,dt_3=0\,, \label{traject42}
\end{equation}
and we can rewrite~\eqref{M42} as follows
\begin{equation*}
M(y)=\Big(t_3-t_2\frac{dt_3}{dt_2}\Big)y_3\,.
\end{equation*}
Again, the expression $M(y)/y_3$ is constant along the trajectory and from the
boundary condition~\eqref{bc} we get the same expression
\begin{equation}
M(y)=\left(\frac{y_1+u}{y_1-u}\right)^{\!p}\!\!y_3\,. \label{solut42}
\end{equation}
Now $u=u(y_1,y_2,y_3)$ is a solution of the equation
\begin{equation}
\frac{y_2-\big(\frac2p-1)y_1}{y_3}=\frac{u-\big(\frac2p-1\big)y_1}{(y_1-u)^p}
\label{extr42}
\end{equation}
that we get from~\eqref{traject42}. As before, we get trajectories ending at
the plane $y_2=-y_1$ not in the whole domain $\Xi_+$, but only in the sector
$-y_1<y_2<\big(\frac2p-1\big)y_1$ (see~Fig.~\ref{4_2}), and
equation~\eqref{extr42} has a unique solution for every point from this sector.
\begin{figure}[ht]
\begin{center}
\begin{picture}(200,200)
\thinlines
\put(20,10){\vector(0,1){180}}
\put(10,100){\vector(1,0){180}}
\put(20,100){\line(1,1){95}}
\put(20,100){\line(1,-1){95}}
\put(20,100){\line(3,1){100}}
\put(30,90){\line(0,1){13.3}}
\put(40,80){\line(0,1){26.6}}
\put(50,70){\line(0,1){40}}
\put(60,60){\line(0,1){53.3}}
\put(70,50){\line(0,1){66.6}}
\put(80,40){\line(0,1){80}}
\put(90,30){\line(0,1){93.3}}
\put(100,20){\line(0,1){106.6}}
\put(110,10){\line(0,1){120}}
\put(120,190){\footnotesize $y_2=y_1$}
\put(120,10){\footnotesize $y_2=-y_1$}
\put(130,140){\footnotesize $y_2=\big(\frac2p-1\big)y_1$}
\put(150,80){\footnotesize $$}
\put(25,190){\footnotesize $y_2$}
\put(195,95){\footnotesize $y_1$}
\end{picture}
\caption{Acceptable sector for the case $4_2\!).$} \label{4_2}
\end{center}
\end{figure}
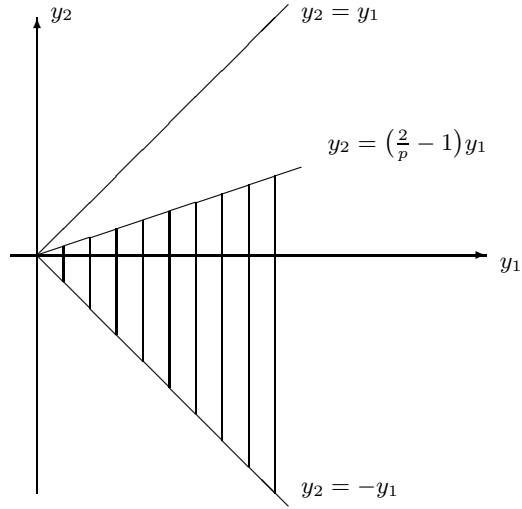
As before, relation~\eqref{uomega} allows us to rewrite the equation of
extremal trajectories~\eqref{extr42} as an implicite expression for $\omega$
(and hence for $M$):
\begin{equation*}
G(x_1,x_2)=G(x_3^{\frac1p},B_{\phantom3}^{\frac1p})\,,
\end{equation*}
or
\begin{equation*}
G(y_1-y_2,y_1+y_2)=y_3G(1,\omega)\,.
\end{equation*}

Now we start the verification which of the obtained solutions satisfies
conditions~\eqref{main_max} or~\eqref{main_min}. We need to calculate $D_i:=
M_{y_iy_i}M_{y_3y_3}-M_{y_iy_3}^2$, $i=1,2$, in four cases
\begin{itemize}
\item[$3_1\!)$]$G(y_1+y_2,-y_1+y_2)=y_3G(\omega,-1)$;
\item[$4_1\!)$]$G(y_1-y_2,-y_1-y_2)=y_3G(1,-\omega)$;
\end{itemize}
\begin{equation}
\label{32}
3_2\!)  \,\,G(y_1+y_2,y_1-y_2)=y_3G(\omega,1);
\end{equation}
\begin{equation}
\label{42}
4_2\!)\,\,G(y_1-y_2,y_1+y_2)=y_3G(1,\omega),
\end{equation}
where $M=y_3\omega^p$. In all situations we have a relation of the form
\begin{equation*}
\Phi(\omega)=\frac{H(y_1,y_2)}{y_3}\,.
\end{equation*}
Till some moment in the future we will not specify the expression for $\Phi$ and $H$, as well
as  for their derivatives, and plug in the specific expression only in the final
result after numerous cancellation. In particular, we introduce
\begin{equation*}
R_1=R_1(\omega):=\frac1{\Phi'}\quad\text{and}\quad
R_2=R_2(\omega):= R'_1=-\frac{\Phi''}{{\Phi'}^2}.
\end{equation*}
We would like to mention here that this idea, allowing us to make calculation shorter,
is taken from the original paper of Burkholder~\cite{Bu1}.

First of all we calculate the partial derivatives of $\omega$:
\begin{align*}
\Phi'\omega_{y_{_3}}=-\frac H{y_3^2}\quad
&\implies\quad\omega_{y_{_3}}=-\frac{R_1H}{y_3^2}\,,
\\
\Phi'\omega_{y_{_i}}=\frac{H_{y_{_i}}}{y_3}\quad
&\implies\quad\omega_{y_{_i}}=\frac{R_1H_{y_{_i}}}{y_3}=\frac{R_1H'}{y_3},\qquad
i=1,2\,.
\end{align*}
Here and further we shall use notation $H'$ for any partial derivative
$H_{y_i}$, $i=1,2$. This cannot cause misunderstanding because only one $i$
participate in calculation of Hessian determinants $D_i$. Moreover, we shall not  mention anymore that
the index $i$ can take two values either $i=1$ or $i=2$.
\begin{align*}
\omega_{y_{_3}y_{_3}}&=-\frac{R_2\omega_{y_{_3}}H}{y_3^2}+2\frac{R_1H}{y_3^3}
=\frac{R_1H}{y_3^4}(R_2H+2y_3)\,,
\\
\omega_{y_{_3}y_{_i}}&=-\frac{R_2\omega_{y_{_i}}H}{y_3^2}-\frac{R_1H'}{y_3^2}
=-\frac{R_1H'}{y_3^3}(R_2H+y_3)\,,
\\
\omega_{y_{_i}y_{_i}}&=\frac{R_2\omega_{y_{_i}}H'}{y_3}+\frac{R_1H''}{y_3}
=\frac{R_1}{y_3^2}(R_2(H')^2+y_3H'')\,.
\end{align*}

Now we pass to the calculation of derivatives of $M=y_3\omega^p$:
\begin{align*}
M_{y_{_3}}&=py_3\omega^{p-1}\omega_{y_{_3}}+\omega^p\,,
\\
M_{y_{_i}}&=py_3\omega^{p-1}\omega_{y_{_i}}\,;
\end{align*}
\begin{align}
M_{y_{_3}y_{_3}}&=py_3\omega^{p-1}\omega_{y_{_3}y_{_3}}+
2p\omega^{p-1}\omega_{y_{_3}}+p(p-1)y_3\omega^{p-2}\omega_{y_{_3}}^2
\notag
\\
\label{M33}
&=\frac{p\omega^{p-2}R_1H^2}{y_3^3}[\omega R_2+(p-1)R_1]\,,
\\
\notag
M_{y_{_3}y_{_i}}&=py_3\omega^{p-1}\omega_{y_{_3}y_{_i}}+
p\omega^{p-1}\omega_{y_{_i}}+p(p-1)y_3\omega^{p-2}\omega_{y_{_3}}\omega_{y_{_i}}
\\
\notag
&=-\frac{p\omega^{p-2}R_1HH'}{y_3^2}[\omega R_2+(p-1)R_1]\,,
\\
\notag
M_{y_{_i}y_{_i}}&=py_3\omega^{p-1}\omega_{y_{_i}y_{_i}}
+p(p-1)y_3\omega^{p-2}\omega_{y_{_i}}^2
\\
\notag
&=\frac{p\omega^{p-2}R_1}{y_3^{\phantom3}}\left([\omega
R_2+(p-1)R_1](H')^2+\omega y_3H''\right)\,.
\end{align}
This yields
\begin{equation}
\label{Di} D_i=M_{y_{_3}y_{_3}}M_{y_{_i}y_{_i}}-M_{y_{_3}y_{_i}}^2=
\frac{p^2\omega^{2p-3}R_1^2H^2H''}{y_3^3}[\omega R_2+(p-1)R_1]\,.
\end{equation}

Notice, that $H'$ disappeared completely.

Now we need to calculate second derivatives of
\begin{equation*}
H(y_1,y_2)=G(\alpha_1y_1+\alpha_2y_2,\beta_1y_1+\beta_2y_2),
\end{equation*}
where
$\alpha_i,\beta_i=\pm1$. And
\begin{align*}
H''&=\frac{\partial^2}{\partial y_i^2}
G(\alpha_1y_1+\alpha_2y_2,\beta_1y_1+\beta_2y_2)
\\
&=\alpha_i^2G_{\!z_{_1}\!z_{_1}}\!\!+2\alpha_i\beta_iG_{\!z_{_1}\!z_{_2}}
\!\!+\beta_i^2G_{\!z_{_2}\!z_{_2}}
\\
&=G_{\!z_{_1}\!z_{_1}}\!\!+G_{\!z_{_2}\!z_{_2}}\!\!\pm2G_{\!z_{_1}\!z_{_2}}\,,
\end{align*}
where the ``$+$'' sign has to be taken if the coefficients in front of $y_i$
are equal and the ``$-$'' sign in the opposite case.

The derivatives of $G$ are simple:
\begin{align*}
G_{\!z_1}\!&=p(z_1+z_2)^{p-2}\big[z_1-(p-2)z_2\big]\,,
\\
G_{\!z_2}\!&=-p(p-1)z_2(z_1+z_2)^{p-2}\,;
\end{align*}
\begin{align*}
G_{\!z_{_1}\!z_{_2}}\!&=p(p-1)(z_1+z_2)^{p-3}\big[z_1-(p-3)z_2\big]\,,
\\
G_{\!z_{_1}\!z_{_2}}\!&=-p(p-1)(p-2)z_2(z_1+z_2)^{p-3}\,,
\\
G_{\!z_{_2}\!z_{_2}}\!&=-p(p-1)(z_1+z_2)^{p-3}\big[z_1+(p-1)z_2\big]\,.
\end{align*}

Note that
$G_{\!z_{_1}\!z_{_1}}\!\!+G_{\!z_{_2}\!z_{_2}}\!\!=2G_{\!z_{_1}\!z_{_2}}$, and
therefore, $H''=4G_{\!z_{_1}\!z_{_2}}$ if $\alpha_i=\beta_i$ and $H''=0$ if
$\alpha_i=-\beta_i$. The first case occurs for $H_{\!y_{_2}\!y_{_2}}$ in
cases~$3_1\!)$, $4_1\!)$ and for $H_{\!y_{_1}\!y_{_1}}$ in cases~$3_2\!)$,
$4_2\!)$. The second case occurs for $H_{\!y_{_1}\!y_{_1}}$ in cases~$3_1\!)$,
$4_1\!)$ and for $H_{\!y_{_2}\!y_{_2}}$ in cases~$3_2\!)$, $4_2\!)$. In fact,
we know that the equality $D_i=0$ has to be fulfilled in the cases~$3_i\!)$
and~$4_i\!)$, because it is just the Monge--Amp\`ere equation we have been
solving.

So we have
\begin{align*}
3_1\!)\quad &z_1=y_1+y_2,&&
\\
&z_2=-y_1+y_2, \qquad
&G_{\!z_{_1}\!z_{_2}}&=p(p-1)(p-2)(y_1-y_2)(2y_2)^{p-3}\,,
\\
4_1\!)\quad &z_1=y_1-y_2,&&
\\
&z_2=-y_1-y_2, \qquad
&G_{\!z_{_1}\!z_{_2}}&=p(p-1)(p-2)(y_1+y_2)(-2y_2)^{p-3}\,,
\\
3_2\!)\quad &z_1=y_1+y_2,&&
\\
&z_2=y_1-y_2, \qquad
&G_{\!z_{_1}\!z_{_2}}&=-p(p-1)(p-2)(y_1-y_2)(2y_1)^{p-3}\,,
\\
4_2\!)\quad &z_1=y_1-y_2,&&
\\
&z_2=y_1+y_2, \qquad
&G_{\!z_{_1}\!z_{_2}}&=-p(p-1)(p-2)(y_1+y_2)(2y_1)^{p-3}\,.
\end{align*}

\bigskip

In the first pair of cases we have $\sign G_{z_1z_2}=\sign H''=\sign(p-2)$ and the opposite sign
in the second pair of cases. In the first pair of cases we have that this sign is the sign of the Hessian determinant $D_2$   up to the $\sign [ wR_2(p-1)R_1]$
(and $D_1=0$ identically); in the second pair of cases we have this sign is the sign of the Hessian determinant $D_1$   up to the $\sign [ wR_2(p-1)R_1]$
(and $D_2=0$ identically).

\bigskip

By the way, we call the attention of the reader to the fact, that, for example, in $4_1)$ above we have necessarily $y_2<0$ (here $y_2$ is fixed and our extremal trajectories in the plane $(y_1,y_3)$ here hit $y_1=-y_2=|y_2|$ as we are always under restrictions $-y_1\le y_2\le y_1$, that is $y_1\ge |y_2|$), so $(-2y_2)^{p-3}$ makes a perfect sense.
The same type of observation holds for all other cases.

\medskip

To complete the investigation of $\sign D_i$ we need to calculate the sign of
the expression in the brackets in~\eqref{Di}:
\begin{equation}
\omega R_2+(p-1)R_1=R_1^2[(p-1)\Phi'-\omega\Phi'']
\end{equation}
\begin{align*}
3_1\!)&\qquad\ &\Phi(\omega)&=G(\omega,-1),
\\
&&\Phi'(\omega)&=G_{\!z_{_1}}(\omega,-1)=p(\omega-1)^{p-2}(\omega+p-2),
\\
&&\Phi''(\omega)&=G_{\!z_{_1}\!z_{_1}}(\omega,-1)=p(p-1)(\omega-1)^{p-3}(\omega+p-3),
\\
&&\hskip-100pt(p-1)\Phi'-\omega\Phi''&=-p(p-1)(p-2)(\omega-1)^{p-3};\rule{0pt}{15pt}
\\
4_1\!)&\qquad\ &\Phi(\omega)&=G(1,-\omega),\rule{0pt}{15pt}
\\
&&\Phi'(\omega)&=-G_{\!z_{_2}}(1,-\omega)=-p(p-1)\omega(1-\omega-1)^{p-2},
\\
&&\Phi''(\omega)&=G_{\!z_{_2}\!z_{_2}}(1,-\omega)=p(p-1)(1-\omega)^{p-3}[1-(p-1)\omega],
\\
&&\hskip-100pt(p-1)\Phi'-\omega\Phi''&=-p(p-1)(p-2)\omega(1-\omega)^{p-3};\rule{0pt}{15pt}
\\
3_2\!)&\qquad\ &\Phi(\omega)&=G(\omega,1),\rule{0pt}{15pt}
\\
&&\Phi'(\omega)&=G_{\!z_{_1}}(\omega,1)=p(\omega+1)^{p-2}(\omega-p+2),
\\
&&\Phi''(\omega)&=G_{\!z_{_1}\!z_{_1}}(\omega,1)=p(p-1)(\omega+1)^{p-3}(\omega-p+3),
\\
&&\hskip-100pt(p-1)\Phi'-\omega\Phi''&=-p(p-1)(p-2)(\omega+1)^{p-3};\rule{0pt}{15pt}
\\
4_2\!)&\qquad\ &\Phi(\omega)&=G(1,\omega),\rule{0pt}{15pt}
\\
&&\Phi'(\omega)&=G_{\!z_{_1}}(1,\omega)=-p(p-1)\omega(\omega+1)^{p-2},
\\
&&\Phi''(\omega)&=G_{\!z_{_1}\!z_{_1}}(\omega,-1)=-p(p-1)(\omega-1)^{p-3}[1+(p-1)\omega],
\\
&&\hskip-100pt(p-1)\Phi'-\omega\Phi''&=-p(p-1)(p-2)\omega(\omega+1)^{p-3}\,.\rule{0pt}{15pt}
\end{align*}
By the way, we call the attention of the reader to the fact, that, for example,
in $4_1)$ above we have necessarily $y_2<0$, so by \eqref{signy2} $\omega< 1$,
so $(1-\omega)^{p-3}$ is fine there. The same type of observation works for
other cases above. We see that in all cases $\sign[(p-1)\Phi'-\omega\Phi'']=
-\sign(p-2)$. Therefore in the first two cases we have $D_2<0$ and this
solution satisfies neither requirement~\eqref{main_max} nor
requirement~\eqref{main_min}. In the second two cases we have $D_1>0$, and the
function $M$ can be a candidate either for $\M_{\max}$ or for $\M_{\min}$
depending on the sign of the second derivative $M_{y_{_3}y_{_3}}$.

Recall that (see~\eqref{M33})
\begin{align*}
M_{y_{_3}y_{_3}}&=\frac{p\omega^{p-2}R_1H^2}{y_3^3}[\omega R_2+(p-1)R_1]\,.
\end{align*}
In the case $3_2\!)$ we have $\sign[\omega R_2+(p-1)R_1]=-\sign (p-2)$, and
therefore we need only to know $\sign R_1=\sign\Phi'=
\sign\frac{d}{d\omega}G(\omega,1)$ Since this solution is considered only in
the sector $\frac{p-2}py_1<y_2<y_1$ (see Fig.~\ref{3_2}), we have
\begin{equation}
\label{G3_2}
G(y_1+y_2,y_1-y_2)=(2y_1)^{p-1}[py_2-(p-2)y_1]>0\,,
\end{equation}
and $\omega$, being the unique positive solution of the equation
\begin{equation}
\label{solut3_2}
G(\omega,1)=(\omega+1)^{p-1}[\omega-p+1]=\frac1{y_3}G(y_1+y_2,y_1-y_2)\,,
\end{equation}
satisfies the condition $\omega>p-1$. Therefore, $\sign R_1=
\sign\frac{d}{d\omega}G(\omega,1) =\sign p (\omega+1)^{p-2}(\omega-p+2)>0$, and
so $\sign M_{y_{_3}y_{_3}}=-\sign(p-2)$, i.e., for $p>2$ this is candidate for
$\M_{\max}$ and for $p<2$ this is candidate for $\M_{\min}$.

\medskip

We are still considering the case $3_2)$. Recall that this function is defined not in the whole domain $\Xi_+$, but only
in the sector $\frac{p-2}py_1<y_2<y_1$. To get a solution everywhere we need to
``glue'' this solution with that we obtained considering the case~2)
(see~\eqref{solut2M}):
\begin{equation}
\label{solut2M1}
M(y)=(y_1+y_2)^p+C(y_3-(y_1-y_2)^p)\,.
\end{equation}
To glue this solution along the plane $y_2=\frac{p-2}py_1$ with that we just
obtained, let us require from the resulting function to be continuous
everywhere. From~\eqref{solut3_2} and~\eqref{G3_2} we see that $G(\omega,1)=0$
on this plane. Therefore, $\omega=p-1$ and $M=\omega^py_3=(p-1)^py_3$. The same
value has solution~\eqref{solut2M1} on this plane for $C=(p-1)^p$.

Now we need to check that we get correct continuation in the sense that if the
solution satisfies~\eqref{main_max}, then its continuation satisfies the same
condition as well, if the solution satisfies~\eqref{main_min}, then the same is
true for its continuation. The Hessian determinants will have the right sign
automatically (actually $D_2=0$ identically). We need only to check the sign of
\begin{equation*}
M_{y_{_1}y_{_1}}=M_{y_{_2}y_{_2}}=
p(p-1)\big((y_1+y_2)^{p-2}-(p-1)^p(y_1-y_2)^{p-2}\big)
\end{equation*}
in the domain $-y_1<y_2<\frac{p-2}py_1$, or in the initial coordinates
$0<x_2<(p-1)x_1$.

For $p>2$ we have
\begin{align*}
(y_1+y_2)^{p-2}=x_2^{p-2}&<(p-1)^{p-2}x_1^{p-2}
\\
&<(p-1)^px_1^{p-2}=(p-1)^p(y_1-y_2)^{p-2}\,,
\end{align*}
and for $p<2$ we have
\begin{align*}
(y_1+y_2)^{p-2}=x_2^{p-2}&>(p-1)^{p-2}x_1^{p-2}
\\
&>(p-1)^px_1^{p-2}=(p-1)^p(y_1-y_2)^{p-2}\,.
\end{align*}
This means that $M$ is a candidate for $\M_{\max}$ if $p>2$ and a candidate for
$\M_{\min}$ if $p<2$, as it has to be.
\begin{equation*}
\end{equation*}

Let us rewrite expression~\eqref{solut2M1} in the same form, as it was made
in~\eqref{solut3_2}.
\begin{equation}
M-Cy_3=(y_1+y_2)^p-C(y_1-y_2)^p=x_2^p-Cx_1^p\,.
\label{symmeq3}
\end{equation}
Therefore, if we change a bit the definition of $G$ defining it on the quadrant
$z_i\ge0$ as follows
\begin{equation}
\label{Gp}
G_p(z_1,z_2)=
\begin{cases}
z_1^p-(p-1)^pz_2^p,&\quad\text{if } z_1\le(p-1)z_2\,,
\\
(z_1+z_2)^{p-1}\big[z_1-(p-1)z_2\big],&\quad\text{if } z_1\ge(p-1)z_2\,,
\end{cases}
\end{equation}
then we can write two our solutions $M$ on $\Xi_+$ in an implicit form as
before:
\begin{equation*}
G(y_1+y_2,y_1-y_2)=y_3G(\omega,1)\,.
\end{equation*}
or solutions $B$ on $\Omega_+$
\begin{equation}
\label{Ge32}
G(x_2,x_1)=G(B_{\phantom3}^{\frac1p},x_3^{\frac1p})\,,
\end{equation}

In the case $4_2)$ we again consider exactly the same $G_p$ from~\eqref{Gp}. In
a similar way we can glue continuously the solution in case $4_2\!)$ found in
the sector $-y_1<y_2<\frac{2-p}py_1$
\begin{equation}
\label{solut4_2}
G(1,\omega)=(\omega+1)^{p-1}[1-(p-1)\omega]=\frac1{y_3}G(y_1-y_2,y_1+y_2)\,,
\end{equation}
which is the same as
\begin{equation}
\label{Ge42}
G(x_1,x_2)=G(x_3^{\frac1p}, B_{\phantom3}^{\frac1p})\,,
\end{equation}
with the solution~\eqref{solut2M1} along the line $y_2=\frac{2-p}py_1$. Here we
have to take $C=(p'-1)^p$, because on the line $y_2=\frac{2-p}py_1$ we have
$G(1,\omega)=0$, i.e., $\omega=p'-1$. Now, in the sector
$-y_1<y_2<\frac{2-p}py_1$ we have
\begin{equation*}
M_{y_{_3}y_{_3}}=
\frac{p\omega^{p-2}R_1^2H^2}{y_3^3}\cdot\frac{p-2}{\omega+1}\,.
\end{equation*}
Therefore, $\sign M_{y_{_3}y_{_3}}=\sign(p-2)$, i.e., for $p<2$ this is
candidate for $\M_{\max}$ and for $p>2$ this is candidate for $\M_{\min}$.

In the ``dual'' sector $x_2>(p'-1)x_1$ (or $y_2>\frac{2-p}py_1$) for $p>2$ we
have
\begin{align*}
(y_1+y_2)^{p-2}=x_2^{p-2}&>(p'-1)^{p-2}x_1^{p-2}
\\
&>(p'-1)^px_1^{p-2}=(p'-1)^p(y_1-y_2)^{p-2}\,,
\end{align*}
and for $p<2$ we have
\begin{align*}
(y_1+y_2)^{p-2}=x_2^{p-2}&<(p'-1)^{p-2}x_1^{p-2}
\\
&<(p'-1)^px_1^{p-2}=(p'-1)^p(y_1-y_2)^{p-2}\,.
\end{align*}
This means that $M$ is a candidate for $\M_{\max}$ if $p<2$ and a candidate for
$\M_{\min}$ if $p>2$.

Using the same ``generalized'' definition~\eqref{Gp} of the function $G$ we can
write  our solutions $M$ on $\Xi_+$ in an implicit form as before:
\begin{equation*}
G(y_1-y_2,y_1+y_2)=y_3G(1,\omega)\,.
\end{equation*}
or solutions $B$ on $\Omega_+$
\begin{equation}
\label{Ge42max}
G(x_1,x_2)=G(x_3^{\frac1p},B_{\phantom3}^{\frac1p})\,,
\end{equation}
which should give, as we said above, the candidate for $\Bel_{\max}$ for $p<2$
and $\Bel_{\min}$ for $p>2$. Notice that for $p>2$ the candidate for, say,
$\Bel_{\max}$ is given by equation~\eqref{Ge32}.

\medskip

It is a bit inconvenient to use one equation for, say, $B_{\max}$ if $p>2$
(this will be~\eqref{Ge32}), and another one (this will be~\eqref{Ge42max})
for the same $\Bel_{\max}$ if $p<2$. We note that after interchanging role of
$z_i$ and replacing $p$ by $p'$ we get the scalar multiple of the original
expression in both lines of~\eqref{Gp}. This allows us to give one expression
for $B_{\max}$ for all $p$ using notation of $p^*=\max\{p,p'\}$. In such a way
we come to formula~\eqref{Fp} for $F_p$, where we introduce additional scalar
coefficients to make this function not only continuous but $C^1$-smooth
everywhere in $\Omega_+$. This smoothness guarantee us that the solution $B$ is
$C^1$-smooth as well.

\section{Proof of Theorem~\ref{t1}. Verification theorem.}

Exactly in the spirit of Stochastic Optimal Control theory we wrote the
PDE~\eqref{MA}, we solved it in the previous section by building $B$ which
solves the equations of Theorem~\ref{t1} (these are the same equations
as~\eqref{Ge32}, \eqref{Ge42}). Now continuing in the spirit of general results
of Stochastic Optimal Control theory~\cite{FR}, \cite{WF} we need to prove that
these solutions in fact are equal to $\Bel_{\max}$, $\Bel_{\min}$. In
Stochastic Optimal Control theory such proofs are called {\it verification
theorems}, and they state roughly that if the solutions have a certain
smoothness (often even slightly less than $C^2$), and if the domain is convex,
then we are fine.

From now on we denote by $B_{\max}$ the unique positive solution of the
equation $F(|x_2|,|x_1|)=F(B_{\phantom3}^{\frac1p},x_3^{\frac1p})$ and by
$B_{\min}$ the unique positive solution of the equation
$F(|x_1|,|x_2|)=F(x_3^{\frac1p},B_{\phantom3}^{\frac1p})$, where the function
$F=F_p$ is defined in~\eqref{Fp}. Existence and uniqueness of the solution
follows from the fact that $F(z_1,z_2)$ is strictly increasing in $z_1$ from
$-p^{*(p-1)}(p^*-1)^pz_2^p$ till $+\infty$ as $z_1$ runs from $0$ to $+\infty$
and it is strictly decreasing in $z_2$ from $p(p^*-1)^{p-1}z_2^p$ till
$-\infty$ as $z_2$ runs from $0$ to $+\infty$. Indeed, the first partial
derivatives of $F$ are
\begin{equation}
\label{Fpz1}
F_{z_{_1}}\!\!=\!
\begin{cases}
pz_1^{p-1},&\hskip-120pt\text{if } z_1\le(p^*-1)z_2\,,
\\
p(1-\frac1{p^*})^{p-1}(z_1+z_2)^{p-2}\big[pz_1-\big((p-1)(p^*-1)-1\big)z_2\big],
\rule{0pt}{25pt}
\\
&\hskip-120pt\text{if } z_1\ge(p^*-1)z_2\,;
\end{cases}
\end{equation}
\begin{equation}
\label{Fpz2}
F_{z_{_2}}\!\!=\!
\begin{cases}
-(p^*-1)^ppz_2^{p-1},&\hskip-100pt\text{if } z_1\le(p^*-1)z_2\,,
\\
-p(1-\frac1{p^*})^{p-1}(z_1+z_2)^{p-2}\big[(p^*-p)z_1+p(p^*-1)z_2\big],
\rule{0pt}{25pt}
\\
&\hskip-100pt\text{if } z_1\ge(p^*-1)z_2\,.
\end{cases}
\end{equation}
Note that both derivatives are continuous everywhere (even at the origin, where
they vanish). Moreover, $F_{z_1}>0$ if $z_1>0$ and $F_{z_2}<0$ if $z_2>0$, i.e.,
$F$ is strictly increasing in $z_1$ and strictly decreasing in $z_2$.

In the case of $B_{\max}$ we look for a solution of the equation
\begin{equation*}
F(B^{\frac1p}_{\phantom3}\!,x_3^{\frac1p})=F(|x_2|,|x_1|)
\end{equation*}
or
\begin{equation*}
F(\omega,1)=\frac1{x_3}F(|x_2|,|x_1|)\,.
\end{equation*}
Thus, we get a continuous solution $\omega(x)$ everywhere except the plane
$x_3=0$, where $\omega$ is not defined. But we can easily estimate the behavior
of $\omega$ nearly the line $x_3=x_1=0$. Since $F$ is decreasing in $z_2$ and
$0\le|x_1|\le x_3^{\frac1p}$, we have
\begin{equation*}
F\big(\frac{|x_2|}{\phantom{a}x_3^{1/p}},1\big)\le
F(\omega,1)=F\big(\frac{|x_2|}{\phantom{a}x_3^{1/p}},\frac{|x_1|}{\phantom{a}x_3^{1/p}}\big)\le
F\big(\frac{|x_2|}{\phantom{a}x_3^{1/p}},0\big)\,.
\end{equation*}
Since $F$ is increasing in $z_1$, we get
\begin{equation*}
\frac{|x_2|}{\phantom{a}x_3^{1/p}}\le \omega\le\omega_0\,,
\end{equation*}
where $\omega_0$ is the solution of the equation
\begin{equation*}
(\omega_0+1)^{p-1}(\omega_0-p^*+1)=\frac{|x_2|^p}{x_3}\,.
\end{equation*}
Whence $\omega_0\ge p^*-1$ and
\begin{equation*}
(\omega_0-p^*+1)^p\le(\omega_0+1)^{p-1}(\omega_0-p^*+1)=\frac{|x_2|^p}{x_3}\,,
\end{equation*}
i.e.,
\begin{equation*}
\omega_0\le p^*-1+\frac{|x_2|}{\phantom{a}x_3^{1/p}}\,,
\end{equation*}
Therefore, for $B=\omega^px_3$ we have the following estimate
\begin{equation*}
|x_2|^p\le B\le\big(|x_2|+(p^*-1)x_3^{1/p}\big)^p\,,
\end{equation*}
which gives the continuity near $x_3=0$. Thus, the solution $B_{\max}$ is
continuous in the closed domain $\Omega$.

Similar considerations gives us the continuity of $B_{\min}$. In that case we
have the equation
\begin{equation*}
F(1,\omega)=\frac1{x_3}F(|x_1|,|x_2|)\,,
\end{equation*}
and hence
\begin{equation*}
F\big(0,\frac{|x_2|}{\phantom{a}x_3^{1/p}}\big)\le
F(1,\omega)=F\big(\frac{|x_1|}{\phantom{a}x_3^{1/p}},\frac{|x_2|}{\phantom{a}x_3^{1/p}}\big)\le
F\big(1,\frac{|x_2|}{\phantom{a}x_3^{1/p}}\big)\,.
\end{equation*}
Now $F(1,\omega)$ is decreasing in $\omega$, therefore,
\begin{equation*}
\frac{|x_2|}{\phantom{a}x_3^{1/p}}\le \omega\le\omega_0\,,
\end{equation*}
where $\omega_0$ is the solution of the equation
\begin{equation*}
1-(p^*-1)^p\omega_0^p=-(p^*-1)^p\frac{|x_2|^p}{x_3}\,,
\end{equation*}
i.e., $\omega_0^p=(p^*-1)^{-p}+|x_2|^p/x_3$ and for $B=\omega^px_3$ we have the
following estimate
\begin{equation*}
|x_2|^p\le B\le |x_2|^p+(p^*-1)^{-p}x_3\,,
\end{equation*}
which gives the continuity near $x_3=0$. Thus, the solution $B_{\min}$ is
continuous in the closed domain $\Omega$ as well.

First step of the proof is to check that the the main inequality
(concavity~\eqref{main-max} for the candidate $B_{\max}$ and
convexity~\eqref{main-min} for the candidate $B_{\min}$) is fulfilled if the
points $x^+, x^-$ satisfy the extra condition on their coordinates:
\begin{equation}
\label{zigzag}
|x_1^+-x_1^-|=|x_2^+-x_2^-|\,.
\end{equation}
This was almost done in the preceding section, when constructing these
candidates. We know that the Hessians of our candidates have the required signs
everywhere in our convex domain $\Omega$ except, possibly, the planes $x_1=0$,
$x_2=0$, and, either $|x_2|=(p^*-1)|x_1|$ for $B_{\max}$ or $|x_1|=
(p^*-1)|x_2|$ for $B_{\min}$. On these hyperplanes our solutions are not
$C^2$-smooth, but this does not prevent them from being correctly convex (for
the $3_2)$, $p>2$ and $4_2)$, $p<2$ cases) and correctly concave for the rest
of the cases (namely, for the $3_2)$, $p<2$ and $4_2)$, $p>2$ cases). This one
checks just by calculating directly the sign of the jump of the derivative.
Namely, one fixes the line $L_t= a+bt$ in the direction of the vector $b=(b_1,
b_2, b_3)$ such that $|b_1|=|b_2|$. We need to prove the concavity of $B$, the
candidate for $\Bel_{\max}$, and the convexity of $B$, the candidate for
$\Bel_{\min}$ on $L_t$. At any point of $L_t$, which is {\it not} the
intersection of $L_t$ with the aforementioned hyperplanes, this concavity
(convexity)  follows from the previous section, this is how the candidates for
$\Bel_{\max}$,  $\Bel_{\min}$ were built in~\eqref{Ge32}, \eqref{Ge42}. At the
points of intersections of $L_t$ with the hyperplanes one can check the sign of
the jump of the derivative of $B(a+tb)$. We leave this as an exercise for the
reader.

Let the triple of points $x,\, x^+\!,\, x^-$ satisfies the following relations
\begin{equation}
\label{splitting}
\begin{aligned}
|x_1^+-x_1^-|=|x_2^+-x_2^-|\,,&\quad x_3^+\ge |x_1^+|^p\,,\quad
x_3^-\ge|x_1^-|^p\,,
\\
x=\alpha^-x^-\!+\alpha^+x^+\,\quad &\alpha^-\!+\alpha^+=1\,, \quad\alpha^\pm>0.
\end{aligned}
\end{equation}
Now we have our solution $B_{\max}$ the following main inequality (biconcavity)
\begin{equation}
\label{MINconc}
B(x)-\alpha^-B(x^-)-\alpha^+B(x^+)\ge 0\,,
\end{equation}
and the opposite main inequality (biconvexity)
\begin{equation}
\label{MINconv}
B(x)-\alpha^-B(x^-)-\alpha^+B(x^+)\le 0
\end{equation}
is true for the solution $B_{\min}$.

\begin{lemma}
\label{est} If a continuous in $\Omega$ function $B$ satisfies the main
inequality~\eqref{MINconc} and the boundary restriction
$B(x_1,x_2,|\,x_1|^p)\ge|\,x_2|^p$, then $B\ge\Bel_{\max}$. If it
satisfies~\eqref{MINconv} and $B(x_1,x_2,|\,x_1|^p)\le|\,x_2|^p$, then
$B\le\Bel_{\min}$.
\end{lemma}

\begin{proof}
Let $I=[0,1]$ and $J$ denote an arbitrary its dyadic subinterval. As always
$J_+, J_-$ are two sons of $J$. Let us fix two bounded measurable test
functions $f,g$ on $I$ such that $|(g, h\ci{J})|=|(f,h\ci{J})|$ for any $J$.
Put
$$
x\!\ci{J}=(\av fJ,\av gJ,\av{\,|f|^p}J)\,.
$$
The fact that  $|(g, h\ci{J})|=|(f,h\ci{J})|$ exactly guarantees that
$x^+\!,\,x^-$ satisfy the assumptions of~\eqref{splitting} and we can rewrite
inequalities~\eqref{MINconc} and~\eqref{MINconv} with $x=x\!\ci{J}$ in the form
\begin{align}
\label{MINconcJ}
|J|\,B(x)-\alpha^-|J_-|\,B(\,x^-)-|J_+|\,\alpha^+B(\,x^+)\ge 0\,,
\\
\label{MINconvJ}
|J|\,B(x)-\alpha^-|J_-|\,B(\,x^-)-|J_+|\,\alpha^+B(\,x^+)\le 0\,.
\end{align}

Let $\J_n$ denotes the set of dyadic subintervals of n-th generation, i.e.,
$\J_0=\{I\}$, and $\J_n$ is the set of suns of elements from $\J_{n-1}$. So,
adding up all our inequalities~\eqref{MINconc} with $x=x\ci{J}$ for
$J\in\J_{n-1}$ we get
$$
\sum_{J\in\J_{n-1}}\!|J|\,B(\,x\!\ci{J})\,\ge\sum_{J\in\J_n}|J|\,B(\,x\!\ci{J})\,.
$$
Adding up these inequality over $n$ from 1 to $N$ we get
$$
B(x)\ge\sum_{J\in\J_N}|J|\,B(\,x\!\ci{J})=\int_0^1\!\!B(x^N\!(t))\,dt\,,
$$
where $x^N\!(t)$ is the step function equal to $x\!\ci{J}$ if $t\in J$ for every $J$ from $\J_N$.

Notice that $\av\vf J\!\to \vf(t)$ almost everywhere when runs over a family of
nested intervals shrinking to the point $t$ for an arbitrary summable function
$\vf$. Therefore, almost everywhere
$$
x^N\!(t)=(\av fJ,\av gJ,\av{\,|f|^p}J) \to (f(t),g(t),|f(t)|^p)\qquad\text{as }
N\to\infty
$$
and since $B$ is continuous, we have
$$
B(x^N\!(t))\to B(f(t),g(t),|f(t)|^p)\ge|g(t)|^p\,.
$$
Now using Lebesgue dominant convergence theorem we come to the estimate
$$
\av{\,|g|^p}I \le B(x)
$$
for every pair of bounded measurable functions $f,g$. And finally approximating
arbitrary $f,g\in L^p(I)$ by its cut-off functions and using monotone
convergence theorem we can extend this inequality to the set of arbitrary
possible test functions $f$ and $g$ what means exactly that $\Bel_{\max}(x)\le
B(x)$.

For the case of $\Bel_{\min}$ in all these considerations we need to change the
sign of inequalities only, and we will get $\Bel_{\min}(x)\ge B(x)$ for $B$
satisfying~\eqref{MINconv}.

\end{proof}

We are left to prove the opposite inequalities
$$
\Bel_{\max}\ge B_{\max}(x)\quad\text{and}\quad\Bel_{\min}\le B_{\min}(x).
$$
This can be done by reversing the reasoning in the lemma above. Using the fact
that domain $\Omega=\{x=(x_1,x_2, x_3): x_3\ge |x_1|^p\}$ is foliated by the
straight line segments (extremal trajectories) it is possible to construct the
sequence of test functions $f_n,g_n$ corresponding any given point $x\in\Omega$
and such that $\av{|g_n|^p}I\to B(x)$. This just supply us with the required
inequality. The reader can see how this type of reasoning is done in
\cite{VaVo2}. The main idea is to travel along the extremal trajectories
starting  from $x\in \Omega$ to build a net $\mathcal{N}:=\{x^+, x^-, x^{++},
x^{+-}, x^{--}, x^{-+},\ldots\}$. All points of the net should belong to
$\Omega$, and we put them on the same extremal trajectory on which $x$ lies for
a while. If one of them, say, $z$ hits the boundary: $\partial \Omega$
(parabola) we stop building children $z^+, z^-$. But then one of them, say,
$\zeta$ can hit the special hyperplanes $x_1=0$ or $x_2=0$. In this case we
choose $\zeta^+, \zeta^-$ in such a way that they lie in different quadrants
very close to $\zeta$. Then we start anew a building of the net for $\zeta^+$
and $\zeta^-$ separately. The closer $\zeta^+, \zeta^-$ are to $\zeta$ the
smaller will be difference $\av{|g_n|^p}I-B(x)$. In such a way for arbitrary
$\ve$ we obtain the inequalities
$$
\Bel_{\max}(x)\ge B_{\max}(x)-\ve\quad\text{and}\quad\Bel_{\min}(x)\le
B_{\min}(x)+\ve\,.
$$

The reader can address to~\cite{VaVo2} to understand how the net $\mathcal{N}$
generates a required pair of functions $f,g$. But in the proof of the following
lemma we only state the result of the described construction that supplies us
with a recursive definition of $f$ and $g$.

\begin{lemma}
\label{oppos-est} The functions $B_{\max}$ and $B_{\min}$ satisfies the
inequalities
$$
\Bel_{\max}(x)\ge B_{\max}(x)\quad\text{and}\quad\Bel_{\min}(x)\le
B_{\min}(x)\,.
$$
\end{lemma}

\begin{proof}
We construct an extremal sequence of pairs $f,g$ for the function
$\Bel_{\max}(x)$ for $p>2$ and for some point $x$ on the plane $x_1=0$. For
$p<2$ the same construction works for $x_2=0$. For all other points we ``glue''
the extremal pairs form the known functions on the ends of the extremal
trajectories. The detailed explanation how to do this can be found e.g.
in~\cite{SV}.

Take an arbitrarily small $\ve$ and recursively define the following pair of
test functions
$$
f(t)=
\begin{cases}
\quad-c\,,&\qquad0<t<\ve\,,
\\
\gamma f(\frac{t-\ve}{1-2\ve})\,,&\qquad\ve<t<1-\ve\,,
\\
\quad\phantom{-}c\,,&\ 1-\ve<t<1
\end{cases}
$$
and
$$
g(t)=
\begin{cases}
\quad d_-\,,&\qquad0<t<\ve\,,
\\
\gamma g(\frac{t-\ve}{1-2\ve})\,,&\qquad\ve<t<1-\ve\,,
\\
\quad d_+\,,&\ 1-\ve<t<1\,.
\end{cases}
$$
where the constants $c$, $d_\pm$, and $\gamma$ will be defined from the
conditions that guarantee that this is an admissible pair of test functions
corresponding to a given point $x=(0,x_2,x_3)$. It is not difficult to see that
these formulas correctly define $f$ and $g$ almost everywhere on $[0,1]$.

It is evident that $x_1=\av f{[0,1]}=0$. Since
$$
\av g{[0,1]}=\ve(d_-+d_+)+(1-2\ve)\gamma\av g{[0,1]}\,,
$$
the condition $\av g{[0,1]}=x_2$ is
\begin{equation}
\label{x_2}
\ve(d_-+d_+)=(1-\gamma+2\ve\gamma)x_2\,.
\end{equation}
The condition $\av{|f|^p}{[0,1]}=x_3$ gives the relation
\begin{equation}
\label{x_3}
2\ve c^p=(1-\gamma^p+2\ve\gamma^p)x_3\,.
\end{equation}
Two more relation we obtain by using condition $|x_1^+-x_1^-|=|x_2^+-x_2^-|$.
Let $t=1-\ve$ be the first splitting point then the condition
$x_1^+-x_1^-=x_2^--x_2^+$ is
\begin{equation}
\label{zigzag1}
c+\frac{\ve c}{1-\ve}=\frac{\ve d_-+(1-2\ve)\gamma
x_2}{1-\ve}-d_+\,.
\end{equation}
The left point $x^+$ is already on the boundary $\partial\Omega$ (the functions
are constants on $I^+$) and we have nothing to split. The left interval $I^-$
we naturally split at the point $t=\ve$. Then the condition
$x_1^+-x_1^-=x_2^+-x_2^-$
\begin{equation}
\label{zigzag2}
c=\gamma x_2-d_-\,.
\end{equation}
From~\eqref{zigzag2} and~\eqref{zigzag1} we get
\begin{equation}
\label{d_pm}
\begin{aligned}
d_-&=\gamma x_2-c\,,
\\
d_+&=\gamma x_2-\frac c{1-2\ve}\,.
\end{aligned}
\end{equation}
Now we can plug in~\eqref{d_pm} into~\eqref{x_2} and obtain in result
\begin{equation}
\label{}
\gamma=1+2\ve\frac{1-\ve}{\;1-2\ve}\cdot\frac c{x_2}\,.
\end{equation}
Let $c_0$ be the limit value of $c$ as $\ve$ tends to 0. (By the way, $c_0$ has
clear geometrical meaning: this is the first coordinate of the end point on
$\partial\Omega$ of the extremal line the second end of which is our initial
$x$.) Then
\begin{align*}
\gamma&\approx 1+2\frac{c_0}{x_2}\ve\,,
\\
\gamma^p&\approx 1+2p\frac{c_0}{x_2}\ve\,,
\\
1-(1-2\ve)\gamma^p&\approx 2(1-\frac{pc_0}{x_2})\ve\,,
\end{align*}
and~\eqref{x_3} turns into equation for $c_0$:
\begin{equation}
\label{c_0}
c_0^p=(1-\frac{pc_0}{x_2})x_3\,,
\end{equation}
which evidently has unique solution $c_0$, $0<c_0<\frac{x_2}p$.
Further we get
\begin{equation}
d_\pm\to x_2-c_0
\end{equation}
and therefore we are able to write down in term of $c_0$ the average we interested in:
\begin{equation}
\label{|g|^p}
\av{\,|g|^p}{[0,1]}=\frac{(d_-^p+d_+^p)\ve}{1-(1-2\ve)\gamma^p}\;\to\;
\frac{2(x_2-c_0)^p}{2(1-\frac{pc_0}{x_2})}=\frac{x_3(x_2-c_0)^p}{c_0^p}\,.
\end{equation}
If we introduce
$$
\omega:=x_3^{-\frac1p}\lim_{\ve\to0}\av{\,|g|^p}{[0,1]}^{\frac1p}\,,
$$
then from~\eqref{|g|^p} we get $c_0=\frac{x_2}{\omega+1}$. We can plug this
expression in~\eqref{c_0} and conclude that $\omega$ is the unique solution of
the equation
$$
x_2^p=(\omega+1)^{p-1}(\omega+1-p)x_3\,,
$$
but this is just the equation $F_p(x_2,0)=F_p(\omega,1)x_3$, whose solution by
definition is
$$
\omega=\Big(\frac{B_{\max}(0,x_2,x_3)}{x_3}\Big)^{\frac1p}\,.
$$
Therefore,
$$
B_{\max}(0,x_2,x_3)=\lim_{\ve\to0}\av{\,|g|^p}{[0,1]}\,,
$$
what yields the desired inequality
$$
B_{\max}(0,x_2,x_3)\le\Bel(0,x_2,x_3)\,.
$$
How to obtain this inequality for arbitrary $x\in\Omega$ was explained in the
beginning of the proof. In the same paper~\cite{SV} the reader can find the
detail explanation how to pass from one splitting to another one (e.g. to a
dyadic family of intervals).

\end{proof}

\section{ Function $u_p$ from function $\Bel$}
\label{shortcut}

We found Burkholder's functions $\Bel_{\max}$ and $\Bel_{\min}$ as claimed in
Theorem~\ref{t1}. As a corollary we immediately we get the sharp constant in
Burkholder's inequality:
\begin{theorem}
\label{pm1}
Let $I=[0,1]$, $\av fI=x_1$, $\av gI =x_2$, $g$ is a Martingale transform of
$f$, and $|\,x_2|\le |\,x_1|$. Then
$$
\av{\,|g|^p}I \le (p^*-1)^p \av{\,|f|^p}I\,.
$$
The constant $p^*-1$, where $p^*:=\max (p, \frac{p}{p-1})$ is sharp.
\end{theorem}

\begin{proof}
We just analyze the form of function $\Bel_{\max}$ from Theorem~\ref{t1} and
immediately see that
$$
\sup_{x\in \Omega,\, |x_2|\le |x_1|} \frac{\Bel_{\max}(x_1,x_2,x_3)}{x_3}\;=\;
(p^*-1)^p\,.
$$

\end{proof}

\begin{theorem}
\label{pm2}
Let $I=[0,1]$, $\av fI=x_1$, $\av gI =x_2$, $g$ is a Martingale transform of
$f$, and $|\,x_2|\le |\,x_1|$. Then
$$
\av{\,|f|^p}I \le (p^*-1)^p \av{\,|g|^p}I\,.
$$
The constant $p^*-1$, where $p^*:=\max (p, \frac{p}{p-1})$ is sharp.
\end{theorem}

\begin{proof}
We just analyze the form of function $\Bel_{\min}$ from Theorem~\ref{t1} and
immediately see that
$$
\inf_{x\in \Omega,\, |x_2|\ge |x_1|} \frac{\Bel_{\min}(x_1,x_2,x_3)}{x_3}\;=\;
(p^*-1)^{-p}\,.
$$

\end{proof}

\noindent{\bf Remark}. The same analysis shows that $\av{\,|g|^p}I\le (p^*-1)^p
\av{\,|f|^p}I$ if and only if  $|\,x_2|\le (p^*-1)|\,x_1|$ in Theorem
\ref{pm1}, and in Theorem~\ref{pm2} $\av{\,|f|^p}I\le(p^*-1)^p\av{\,|g|^p}I$ if
and only if $|\,x_2|\ge (p^*-1)^{-1}|\,x_1|$.

\bigskip

\noindent{\bf Notation.} Below we use $\beta_p:= (p^*-1)^p$. Put
\begin{align*}
\phi_{\max}(x_1,x_2)\ &:=\sup_{x_3:(x_1,x_2,x_3)\in\Omega}\big[\Bel_{\max}(x_1,x_2,x_3)-\beta_p x_3\big]\,,
\\
\phi_{\min}(x_1,x_2)\ &:=\inf_{x_3:(x_1,x_2,x_3)\in\Omega}\big[\Bel_{\min}(x_1,x_2,x_3)-\beta_p^{-1}x_3\big]\,.
\end{align*}

These functions are defined on the whole $\R^2$.

\noindent{\bf Definition.} If for all pairs of points $x^\pm\in\R^2$ such that
\begin{equation}
\label{zz}
|x_1^+-x_1^-|=|x_2^+-x_2^-|\quad\text{and}\quad x=\frac{x^++x^-}2
\end{equation}
the function $\phi$ on $\R^2$ satisfies the condition
\begin{equation}
\label{zzconc}
\phi(x)-\frac{\phi(x^-)+\phi(x^+)}2\ge 0\,,
\end{equation}
then it is called zigzag concave. If the opposite inequality holds
\begin{equation}
\label{zzconv}
\phi(x)-\frac{\phi(x^-)+\phi(x^+)}2\le 0
\end{equation}
the function $\phi$ is called zigzag convex.
The next theorem gives an independent description of $\phi_{\max}$ and $\phi_{\min}$.

\begin{theorem}
\label{lcm} Function $\phi_{\max}$ is the least zigzag concave majorant of the
function $h_{\max}(x):=|\,x_2|^p-\beta_p |\,x_1|^p$. Function $\phi_{\min}$ is
the greatest zigzag convex minorant of the function
$h_{\min}(x):=|\,x_2|^p-\beta_p^{-1} |\,x_1|^p$.
\end{theorem}

\noindent{\bf Remark.} Notice that this is slightly counterintuitive:
$\Bel_{\max}(x)-\beta_p x_3$ is zigzag concave for any fixed $x_3$, and the
supremum of concave functions is {\it not} usually concave. The same is true
about infimum of convex functions.

\begin{proof}
Let $x^\pm$ and $x$ are as in~\ref{zz}. It is obvious that $\phi_{\max}$ is
zigzag concave. One verifies this just by definition. In fact, if for any
$x^-\in \R^2$ we can choose $x_3^-$ such that the supremum in the definition of
$\phi_{\max}$ is almost attained, i.e., $\Bel_{\max}(x^-)-\beta_p x_3^->
\phi_{\max}(x_-)-\ve$ for a given $\ve$ and the same for $x^+\in \R^2$. We
define $x_3=\frac{x_3^-+ x_3^+}2$ and $\tilde x=(x,x_3)$. Then
using~\eqref{MINconc} we can write
\begin{align*}
\phi_{\max}(x)&\ge\Bel_{\max}(\tilde x)-\beta_p x_3
\\
&\ge \frac{\Bel_{\max}(\tilde x^-)+\Bel_{\max}(\tilde x^+)}2
-\beta_p\frac{x_3^-+x_3^+}2
\\
&\ge\frac{\phi_{\max}(x^-)+\phi_{\max}(x^+)}2-\ve\,,
\end{align*}
what yields~\eqref{zzconc}. Inequality~\eqref{zzconv} is totally similar.

\medskip

As $\sup$ is bigger than $\lim$ we conclude
$$
\phi_{\max}(x)\ge\lim_{x_3\to|x_1|^p}\big[\Bel_{\max}(\tilde x)-\beta_p
x_3\big] = |\,x_2|^p-\beta_p |\,x_1|^p = h_{\max}(x)\,.
$$
As $\inf$ is smaller than $\lim$ we get analogously
$$
\phi_{\min}(x) \le \lim_{x_3\to|x_1|^p}\big[\Bel_{\min}(\tilde x)-\beta_p^{-1}
x_3\big] = |\,x_2|^p-\beta_p^{-1} |\,x_1|^p = h_{\min}(x)\,.
$$
This is because the boundary values of $\Bel_{\max}$ and $\Bel_{\min}$ are
$|\,x_2|^p$.

\medskip

We are left to see that $\phi_{\max}$ is the {\it least} such majorant (and a
symmetric claim for $\phi_{\min}$). Let $\psi$ be a zigzag concave function
such that
\begin{equation}
\label{psih}
\phi_{\max}\ge \psi \ge h_{\max}\,.
\end{equation}
Consider function $\Psi(\tilde x):=\psi(x)+\beta_p x_3$. It is immediate that
$\Psi$ satisfies~\eqref{MINconc}. On the boundary of $\Omega$ we have
$\Psi(x)\ge|\,x_2|^p$, this is just by the right hand side of~\eqref{psih}.
Then Lemma~\ref{est} yields
$$
\Psi(\tilde x)\ge\Bel_{\max}(\tilde x)\,.
$$
Then, obviously,
$$
\psi(x)=\sup_{x_3\colon\tilde x\in \Omega}\big[\Psi(\tilde x) -\beta_p x_3\big]
\ge \sup_{x_3\colon\tilde x\in \Omega}\big[\Bel_{\max}(\tilde x)-\beta_p
x_3\big]= \phi_{\max}(x)\,.
$$
So we proved that $\phi_{\max}$ is the least zigzag concave majorant of
$h_{\max}$. Symmetric consideration will bring us the fact that $\phi_{\min}$
is the largest zigzag convex minorant of $h_{\min}$.

\end{proof}

The reader should look now at function $F_p$ from Theorem~\ref{t1}. It would be
interesting to obtain the formulae for $\phi_{\max}$ and $\phi_{\min}$,
especially using this $F_p$. It would be also interesting to understand the
role of function
\begin{equation}
\label{upf}
u_p(x_1,x_2):= p(1-\frac1{p^*})^{p-1} (|x_1|+|x_2|)^{p-1}( |x_2|- (p^*-1)|x_1|)\,,
\end{equation}
mentioned in the introduction and used repeatedly by Burkholder. May be it is
equal to $\phi_{\max}$? The answer is ``no'', but we can prove the following
\begin{theorem}
\label{fla}
\begin{equation}
\label{flaeq}
\phi_{\max}(x_1,x_2) = F_p(|\,x_2|, |\,x_1|)\,.
\end{equation}
\end{theorem}

\begin{proof}
We shall consider only the case $p>2$, the case $p<2$ is similar. Due to the
symmetry with respect to the change of $x_1$ to $-x_1$ and $x_2$ to $-x_2$ it
is enough to check equality~\eqref{flaeq} in the quadrant $x_1>0$, $x_2>0$. If
$x_2\le(p-1)x_1$, we get an explicit formula for $\Bel_{\max}$ from
Theorem~\ref{t1}: $\Bel_{\max}(\tilde x)= x_2^p+(p-1)^p(x_3-x_1^p)$, and
therefore,
$$
\phi_{\max}(x_1,x_2)=\sup_{x_3:\tilde x \in\Omega} \big[\Bel_{\max}(\tilde
x)-(p-1)^p x_3\big]=x_2^p-(p-1)^px_1^p=F_2(x_2,x_1)\,.
$$
So in the rest of the proof we shall consider only the domain
$\{x=(x_1,x_2)\colon 0\le (p-1)x_1<x_2\}$. Moreover, since both functions
$\phi_{\max}$ and $F_p$ are $p$-homogeneous, it is sufficient to
check~\eqref{flaeq} on the interval $S:=\{x\colon 0\le px_1<1, x_1+x_2=1\}$.
(Indeed, the condition $px_1<1$ on the line $x_1+x_2=1$ means $x_2>(p-1)x_1$.)

The function $F_p$ is linear on $S$: $F_p(1-x_1, x_1)=
\frac{(p-1)^{p-1}}{p^{p-2}}(1-px_1)$. Now we check that $\phi_{\max}$ is linear
as well. To this end we check inequality
\begin{equation}
\label{phi-conv}
\phi_{\max}(x)\le\frac{\phi_{\max}(x_1-a,x_2+a)+\phi_{\max}(x_1+a,x_2-a)}2
\end{equation}
for all $x\in S$ and sufficiently small $a$, this just means linearity of
$\phi_{\max}$ on $S$, because the opposite inequality  follows from the zigzag
concavity of $\phi_{\max}$~\eqref{zzconc}.

Fix $x\in S$ and $\ve>0$. Take $x_3$ such that $B(\tilde x) -\beta_p x_3 \ge
\phi_{\max}(x)-\ve$. Due to condition $x_2>(p-1)x_1$ the extremal trajectory
$L_x$ of $\Bel_{\max}$ passing through the point $\tilde x=(x,x_3)$ is not
vertical, it hits at some point the plane $x_1=0$. Therefore we can take two
different points $\tilde x^\pm=(x^\pm,x_3^\pm)$ on $L_x$ such that $\tilde x=
\half(\tilde x^++\tilde x^-)$. We know
three things:
\begin{align*}
\Bel_{\max}(\tilde x) -\beta_p x_3 &\ge \phi_{\max}(x)-\ve\,,
\\
\Bel_{\max}(\tilde x^+) -\beta_p x_3^+ &\le \phi_{\max}(x^+)\,,
\\
\Bel_{\max}(\tilde x^-) -\beta_p x_3^- &\le \phi_{\max}(x^-)\,.
\end{align*}
Since the function $\Bel_{\max}$ is linear along $L_x$, we can write the
following chain of inequalities
\begin{align*}
\phi_{\max}(x)-\ve&\le\Bel_{\max}(\tilde x) -\beta_p x_3
\\
&=\frac{\big[\Bel_{\max}(\tilde x^+)-\beta_p x_3^+\big]+
\big[\Bel_{\max}(\tilde x^-) -\beta_p x_3^-\big]}2
\\
&\le\frac{\phi_{\max}(x^+)+\phi_{\max}(x^-)}2\,.
\end{align*}
Since $\ve$ is arbitrary, we come to the desired convexity~\eqref{phi-conv}.

Function $F_p(x_2, x_1)$ is a concave $C^1$-smooth function majorazing
$h_{\max}$ on $S$. This is immediate from its formula. Functions
$\phi_{\max}(x_1,x_2)$ and $F_p(x_2, x_1)$  are linear on $S$ and at the point
$x=x_p=:(\,\frac1p\,,1\!-\!\frac1p)$ both are equal $h_{\max}(x_p)=0$.
Therefore, to prove that they are identical it is sufficient to check that
their derivatives at $x_p$ along $S$ are equal as well. Since $\phi_{\max}$ is
a majorant of $h_{\max}$ and both functions are equal at $x_p$, then the left
derivative of $\phi_{\max}$ at $x_p$ is not grater than the derivative of
$h_{\max}$ at this point. On the other hand, since $\phi_{\max}$ is the least
majorant it is not grater then $F_p$, i.e., its left derivative at $x_p$ is not
less that the derivative of $F_p$ there, but latter coincides with the
derivative of $h_{\max}$. Hence all three derivatives along $S$ are equal at
the point $x_p$ and we proved $\phi_{\max}(x_1, x_2) = F_p(x_2, x_1)$.

\end{proof}

\begin{theorem}
\label{flamin}
\begin{equation}
\label{flamineq}
\phi_{\min}(x_1,x_2) =- (p^*-1)^{-p}F_p(|\,x_1|, |\,x_2|)\,.
\end{equation}
\end{theorem}

The proof of this Theorem is absolutely similar to the proof of Theorem~\ref{fla}.

\bigskip

Burkholder often used function $u_p$ from~\eqref{upf}. To demystify it let us
notice that it is also $p$-homogeneous and as such can be considered only on
the segment $x_1+x_2=1$, $x_i>0$. On this segment function $u_p$ becomes
linear. It is a majorant of $h_p$, and its graph is tangent to the graph $h_p$
exactly at point $x_p$ on $S$, where $h_p$ vanishes. It is not the least zigzag
concave function greater than $h_p$ (of course not, $\phi_{\max}$ is such), but
it is the least zigzag concave function larger than $h_p$ and such that on all
segments $\{x\colon x_i>0,\, x_1+x_2=\const\}$ it is not only concave, but also
linear. (Keeping in mind the symmetries $x\to -x_1, x_2\to -x_2$ we can
consider the first quadrant only.)

\medskip

This is already proved, and we leave the detailed reasoning to the reader. One
more thing we want to mention is that we could have considered a slightly more
general problem. Namely, instead of majorazing the function $h_{\max}(x_1, x_2)
= |\,x_2|^p-(p^*-1)^p|\,x_1|^p$ we could have started with any function
$$
h_c(x_1, x_2):=|x_2|^p-c\,|x_1|^p\,.
$$
The reader can easily see that we have proved the following theorem (of course
Burkholder already proved the most of it long ago).

\begin{cor}
\label{Hc} The smallest $c$ for which there exists a zigzag concave function
$\phi_c$ majorazing $h_c$ is equal to $(p^*-1)^p$. For this $c$ the least
zigzag majorant is $F_p(|\,x_1|, |\,x_2|)$. The smallest $c$ or which there
exists a zigzag concave function $\phi_c$ majorazing $h_c$ such that it is
linear on the segment $\{x\colon x_i>0,\, x_1+x_2=\const\}$, symmetric and
$p$-homogeneous is equal to $(p^*-1)^p$. For this $c$ the least zigzag majorant
linear on $\{x\colon x_i>0,\, x_1+x_2=\const\}$, symmetric and $p$-homogeneous
is $u_p(x_1, x_2)$.
\end{cor}

\bigskip

\noindent{\bf Remark.} Notice an interesting thing which we do not know how to
explain. Given function $\Bel_{\max}$ from Theorem~\ref{t1} we can easily
diminish the number of variables and construct $\phi_{\max}$. But amazingly we
can also find $\Bel_{\max}$ if only $\phi_{\max}$ is given. In fact, Theorem
\ref{fla} gives the formula for $\phi_{\max}$ via $F_p$. Then $F_p$ allows us
to find $\Bel_{\max}$. If we now combine Theorem~\ref{t1} and
Theorems~\ref{fla}--\ref{flamin} to conclude that
\begin{cor}
\label{Bfromphi} Given a point $x\in \Omega$, if we know $\phi_{\max}$ we can
find $\Bel_{\max}(x)$ by solving equation:
$$
\phi_{\max}(x_1, x_2) = \phi_{\max}(x_3^{\frac1p},\Bel_{\max}^{\frac1p})\,.
$$
Symmetric formula allows to find $\Bel_{\min}$ if $\phi_{\min}$ is known:
$$
\phi_{\min}(x_2, x_1) = \phi_{\min}(\Bel_{\min}^{\frac1p},x_3^{\frac1p})\,.
$$
\end{cor}

\markboth{}{\sc \hfill \underline{References}\qquad}


\begin{thebibliography}{XXXXXX}
\label{rf}

\bibitem[BaJa1]{BaJa1}
{\sc R. Banuelos, P. Janakiraman},
{\em $L^p$-bounds for the  Beurling--Ahlfors transform}. Trans. Amer. Math.
Soc. {\bf 360} (2008), no. 7, 3603--3612.

\bibitem[BaMH]{BaMH}
{\sc R. Banuelos, P. J. Mendez-Hernandez}, {\em Space-time Brownian
motion and the Beurling--Ahlfors transform.} Indiana Univ. Math. J. {\bf 52}
(2003), no. 4, 981--990.

\bibitem[BaWa1]{BaWa1}
{\sc R. Banuelos, G. Wang}, {\em Sharp inequalities for martingales
with applications to the Beurling--Ahlfors and Riesz transforms}, Duke Math.
J., {\bf 80} (1995), 575--600.

\bibitem[Bu]{Bu}
{\sc D.~ Burkholder,} {\em  Martingale transforms,}  Ann. Math. Statist. 37
(1966), 1494-1504.

\bibitem[Bu1]{Bu1}
{\sc D.~Burkholder}, {\em Boundary value problems and sharp estimates for the
martingale transforms}, Ann. of Prob. {\bf 12} (1984), 647--702.

\bibitem[Bu2]{Bu2}
{\sc D.~Burkholder}, {\em An extension of classical martingale inequality},
Probability Theory and Harmonic Analysis, ed. by J.-A. Chao and W. A.
Woyczynski, Marcel Dekker, 1986.

\bibitem[Bu3]{Bu3}
{\sc D.~Burkholder}, {\em Sharp inequalities for martingales and stochastic
integrals}, Colloque Paul L\' evy sur les Processus Stochastiques (Palaiseau,
1987), Ast\'erisque No. 157--158 (1988), 75--94.

\bibitem[Bu4]{Bu4}
{\sc D.~Burkholder}, {\em Differential subordination of harmonic functions and
martingales}, (El Escorial 1987), Lecture Notes in Math., {\bf 1384} (1989),
1--23.

\bibitem[Bu5]{Bu5}
{\sc D.~Burkholder}, {\em Explorations of martingale theory and its
applications}, Lecture Notes in Math. {\bf 1464} (1991), 1--66.

\bibitem[Bu6]{Bu6}
{\sc D.~Burkholder}, {\em Strong differential subordination and stochastic
integration}, Ann. of Prob. {\bf 22} (1994), 995--1025.

\bibitem[Bu7]{Bu7}
{\sc D.~Burkholder}, {\em A proof of the Pe\l czynski's conjecture for the Haar
system}, Studia MAth., {\bf 91} (1988), 79--83.

\bibitem[DV1]{DV1}
{\sc O. Dragicevic, A.~Volberg}  {\em Sharp estimates of the Ahlfors--Beurling
operator via averaging of Martingale transform,} Michigan Math. J. {\bf 51}
(2003), 415-435.

\bibitem[DV2]{DV2}
{\sc O.~Dragicevic, A.~Volberg}, {\em Bellman function, Littlewood--Paley
estimates, and asymptotics of the Ahlfors--Beurling operator in
$L^p(\mathbb{C})$, $p\to\infty$}, Indiana Univ. Math. J. {\bf 54} (2005), no.
4, 971--995.

\bibitem[DV3]{DV3}
{\sc O. Dragicevic, A. Volberg}, {\em Bellman function and dimensionless
estimates of classical and Ornstein--Uhlenbeck Riesz transforms.} J. of Oper.
Theory, {\bf 56} (2006) No. 1, pp. 167--198.

\bibitem[FR]{FR}
{\sc W. Fleming, R. Rishel}, {\em Deterministic and Stochastic Optimal
Control.} Applications of Mathematics. Stochastic Modeling and Applied
Probability, Vol. 1, Springer Verlag, 1998.

\bibitem[WF]{WF}
{\sc W. Fleming}, {\em Controlled Markov processes and viscosity solutions of
nonlinear evolution}, Publications of the Scuola Normale Superiore, 2007.

\bibitem[HTV]{HTV}
{\sc S. Hukovic, S. Treil, A. Volberg}, {\em  Bellman functions and sharp
weighted estimates for the square functions}. Complex analysis, operators, and
related topics (in memory of S.~A.~Vinogradov), Oper. Theory Adv. Appl., {\bf
113}, Birkhauser, 2000, pp. 97--113.


\bibitem[NT]{NT}
{\sc F. Nazarov, S. Treil}, {\em The hunt for Bellman function: applications to
estimates of singular integral operators and to other classical problems in
harmonic analysis}, Algebra i Analiz {\bf 8} (1997), no. 5, 32-162.

\bibitem[NTV1]{NTV1}
{\sc F.~Nazarov, S.~Treil, A.~Volberg}, {\em The Bellman functions and
two-weight inequalities for Haar multipliers}, J. of Amer. Math. Soc., 12
(1999), 909-928.

\bibitem[NTV2]{NTV2}
{\sc F.~Nazarov, S.~Treil, A.~Volberg}, {\em Bellman function in stochastic
control and harmonic analysis}, Systems, approximation, singular integral
operators, and related topics (Bordeaux, 2000), 393--423, Oper. Theory Adv.
Appl., Vol.~129, Birkhauser, Basel, 2001.

\bibitem[NV]{NV}
{\sc F. Nazarov and A. Volberg}, {\em Bellman function, two weighted Hilbert
transforms and embeddings of the model spaces $K_\theta$}, Dedicated to the
memory of Thomas H. Wolff. J. Anal. Math. {\bf 87} (2002), 385--414.

\bibitem[NV1]{NV1}
{\sc F.~Nazarov and A.~Volberg}, {\em Heating of the Ahlfors--Beurling operator
and estimates of its norm}, St. Petersburg Math. J.,  {\bf 14} (2003) no. 3.

\bibitem[P1]{P1}
{\sc S. Petermichl}, {\em A sharp bound for weighted Hilbert transform in terms
of classical $A_p$ characteristic},  Amer. J. Math. {\bf 129} (2007), no. 5,
1355--1375.

\bibitem[P2]{P2}
{\sc S. Petermichl},  {\em A sharp bound for weighted Riesz transform in terms
of classical $A_p$ characteristic}, Proc. Amer. Math. Soc. 2008.

\bibitem[PV]{PV}
{\sc S. Petermichel, A. Volberg}, {\em Heating the Beurling operator:  weakly
quasiregular maps on the plane are quasiregular,} Duke Math. J.,  {\bf 112}
(2002), no.2, pp. 281--305.

\bibitem[Pog]{Pog}
{\sc A. V. Pogorelov}, {\em Extrinsic geometry of convex surfaces}.
Translations of Mathematical Monographs, Amer. Math. Soc., v.~35, 1973.

\bibitem[SlSt]{SlSt}
{\sc L. Slavin, A. Stokolos}, {\em The maximal operator on $L^p(\R^n)$},
Preprint, 2007.

\bibitem[SV]{SV}
{\sc L. Slavin, V. Vasyunin}, {\em Sharp results in the integral-form
John--Nirenberg inequality}, Trans. Amer. Math. Soc. (to appear); Preprint
arXiv:0709.4332v1, 2007; (http://arxiv.org/abs/0709.4332).

\bibitem[Str]{Str} {\sc D. Strook}, {\em Probability Theory, an Analytic View},
Cambridge University Press, 1993.

\bibitem[VoEcole]{VoEcole} {\sc A. Volberg}, {\em Bellman approach
to some problems in Harmonic Analysis}, S\'eminaires des Equations
aux deriv\'ees partielles. Ecole Polit\'echnique, 2002, expos\'e
XX, pp. 1--14.

\bibitem[VaVo]{VaVo}
{\sc  V. Vasyunin, A. Volberg}, {\em Bellman functions technique in Harmonic Analysis},
sashavolberg.wordpress.com

\bibitem[VaVo1] {VaVo1}
{\sc V. Vasyunin, A. Volberg}, {\em The Bellman function for certain two weight
inequality: the case study}, St. Petersburg Math. J. {\bf 18} (2007), no. 2,
pp. 201--222.

\bibitem[VaVo2] {VaVo2}
{\sc V. Vasyunin, A. Volberg}, {\em Monge--Amp\`ere equation and Bellman
optimization of Carleson Embedding Theorems},  Amer. Math. Soc. Transl. (2),
vol.~226, ``Linear and Complex Analysis'', 2009, pp. 195--238.
(arXiv:0803.2247)

\bibitem[VaVo3] {VaVo3}
{\sc V. Vasyunin, A. Volberg}, {\em Bellster and others}, Preprint, 2008.

\end{thebibliography}
\end{document}